\PassOptionsToPackage{nameinlink,capitalize}{cleveref}

\documentclass{article} %
\usepackage[dvipsnames,table,xcdraw]{xcolor} %
\usepackage{iclr2022_conference,times}

\usepackage{hyperref}
\usepackage{url}

\makeatletter
\DeclareRobustCommand{\iscircle}{\mathord{\mathpalette\is@circle\relax}}
\newcommand\is@circle[2]{%
  \begingroup
  \sbox\z@{\raisebox{\depth}{$\m@th#1\bigcirc$}}%
  \sbox\tw@{$#1\square$}%
  \resizebox{!}{\ht\tw@}{\usebox{\z@}}%
  \endgroup
}
\makeatother

\newcommand*{\vcenterimage}[2]{\vcenter{\hbox{\includegraphics[#1]{#2}}}}
\newcommand*{\vcenterarrow}{\vcenter{\hbox{$\Longrightarrow$}}}

\newcommand{\rebuttal}[1]{#1}

\usepackage[inline]{enumitem}
\usepackage[%
	thmreset=section,
	eqreset=section,
]{myPreamble}

\renewcommand\bm[1]{%
	\mathchoice%
		{\text{\boldmath{\(\displaystyle #1\)}}}%
		{\text{\boldmath{\(\textstyle #1\)}}}%
		{\text{\boldmath{\(\scriptstyle #1\)}}}%
		{\text{\boldmath{\(\scriptscriptstyle #1\)}}}%
}

	\newcommand{\D}{\mathcal D}

	\newcommand{\z}{\bar z}

	\newcommand{\G}{G}

\makeatletter

	\newcommand{\BregmanKernel}{H}						%
	\newcommand{\h}{\@ifstar{\@@h}{\@h}}				%
	\let\P\relax
	\newcommand{\E}{\@ifstar\@@E\@E}					%
	\newcommand{\P}{\@ifstar\@@P\@P}					%

	\newcommand{\@h}{\@ifnextchar_{{\color{cyan}h}}{{\color{blue}\BregmanKernel}}}
	\newcommand{\@@h}{\@ifnextchar_{{\color{orange}\hat h}}{{\color{red}\hat\BregmanKernel}}}

	\newcommand{\@P}[2][k]{\mathcal P_{#1}\ifstrempty{#2}{}{\left[#2\right]}}
	\newcommand{\@@P}[2][k]{\mathcal P_{#1}\ifstrempty{#2}{}{[#2]}}

	\newcommandx{\@E}[3][1=k,3={}]{
		\mathbb E_{#1}\ifstrempty{#2}{}{
			\left[
				#2\ifstrempty{#3}{}{\vphantom{#3}\right|\left.#3}
			\right]
		}
	}
	\newcommandx{\@@E}[3][1=k,3={}]{
		\mathbb E_{#1}\ifstrempty{#2}{}{
			[#2\ifstrempty{#3}{}{|#3}]
		}
	}

\makeatother

\makeatletter

	\renewcommand{\alglinenumber}[1]{%
		\footnotesize
		\textbf{\thealgorithm}.%
		\fillwidthof[l]{\oldstylenums{88}:}{\oldstylenums{\arabic{ALG@line}}:}%
	}
	\renewcommand{\theALG@line}{\thealgorithm.\oldstylenums{\arabic{ALG@line}}}
	\providecommand{\theHALG@line}{\thealgorithm.\arabic{ALG@line}}

\makeatother

\newtheorem{definition}{Definition}
\newtheorem{proposition}{Proposition}
\newcounter{example}%
\newenvironment{example}[1][]{\refstepcounter{example}\par\medskip
   \noindent {{\bf Example~\theexample}: #1} %
   }{\medskip}

\usepackage{mathabx}
\usepackage{wasysym}

\title{Escaping limit cycles: Global convergence for constrained nonconvex-nonconcave minimax problems}

\author{
    Thomas Pethick\thanks{Laboratory for Information and Inference Systems (LIONS), EPFL (\href{mailto:thomas.pethick@epfl.ch}{thomas.pethick@epfl.ch})} \And 
    Puya Latafat\thanks{Department of Electrical Engineering (ESAT-STADIUS), KU Leuven} \And 
    Panagiotis Patrinos\footnotemark[2] \And 
    Olivier Fercoq\thanks{Laboratoire Traitement et Communication d'Information, Télécom Paris, Institut Polytechnique de Paris} \And
    Volkan Cevher\footnotemark[1]
}

\iclrfinalcopy %
\begin{document}

\maketitle

\begin{abstract}

This paper introduces a new extragradient-type algorithm for a class of nonconvex-nonconcave minimax problems. It is well-known that finding a local solution for general minimax problems is computationally intractable. This observation has recently motivated the study of structures sufficient for convergence of first order methods in the more general setting of variational inequalities when the so-called \emph{weak Minty variational inequality} (MVI) holds.
This problem class captures non-trivial structures as we demonstrate with examples, for which a large family of existing algorithms provably converge to limit cycles. 
Our results require a less restrictive parameter range in the weak MVI compared to what is previously known, thus extending the applicability of our scheme. 
The proposed algorithm is applicable to constrained and regularized problems, and involves an adaptive stepsize allowing for potentially larger stepsizes. Our scheme also converges globally even in settings where the underlying operator exhibits limit cycles.

\end{abstract}

\section{Introduction}
\label{sec:introduction}

Many machine learning applications, from generative adversarial networks (GANs) to robust reinforcement learning, result in
 nonconvex-nonconcave constrained minimax problems,  
 which pose notorious difficulties to the scalable first order methods. Indeed, there is no shortage of results illustrating divergent or cycling behavior when going beyond minimization problems  \citep{benaim1999mixed,hommes2012multiple,mertikopoulos2018cycles,hsieh2021limits}.

Traditionally, minimax problems have been studied for more than half a century under the umbrella of the variational inequalities (VIs).
The extragradient-type algorithms from the VI literature was recently brought to the awareness of the machine learning community \citep{mertikopoulos2018optimistic,gidel2018variational,bohm2020two}, and have provided a principled way of stabilizing training and avoiding Poincaré recursions. %
However, these results mostly concern the convex-concave setting.

In nonconvex-nonconcave minimax problems, or more generally nonmonotone variational inequalities (VIs), even finding a \emph{local} solution is in general intractable. 
This has been made precise through exponential lower bound of the classical optimization type \citep{hirsch1987exponential} and computational complexity results \citep{papadimitriou1994complexity,daskalakis2021complexity}.
This is in sharp contrast to minimization problems, where only finding a \emph{global} solution is intractable. %
The recent result of \citep{hsieh2021limits} provides some intuition behind this difference by showing that the asymptotic limits of most schemes, including extragradient, can converge to attracting limit cycles. %

To make progress in lieu of these negative results, \citet{diakonikolas2021efficient} proposes a simple generalization of extragradient, called \eqref{eq:eg+}, that can converge to a stationary point even for a class of nonmonotone problems provided that the \emph{weak Minty variational inequality} (MVI) holds. This problem class is parametrized by a constant $\rho$, which controls the degree of nonconvexity. 
However, given the range of $\rho$ in \citet{diakonikolas2021efficient}, the new class is still too small to include even the simplest counterexample of \citet{hsieh2021limits} for the general Robbins-Monro schemes. %

\paragraph{Contributions} Building on the analysis in \citet{diakonikolas2021efficient}, we propose a new adaptive scheme, called \eqref{eq:CurvatureEG}, that converges even in the difficult  counter example of \citet{hsieh2021limits} as illustrated in \cref{fig:forsaken}. %
Our main contributions are summarized below.  
\begin{enumerate}[%
  leftmargin=0pt,
  label={\arabic*.},
  align=left,
  labelwidth=1em,
  itemindent=\labelwidth+\labelsep,
]
\item We propose an adaptive extragradient-type algorithm that converges for a larger range of $\rho$, the parameter in the weak MVI assumption (cf. \cref{ass:Minty:Struct}) than previously known. 

\item More importantly, we show that convergence is ensured if $2\rho + \gamma_k>0$, where $\gamma_k$ is the extrapolation stepsize. This is crucial since by selecting $\gamma_k$ through a backtracking procedure larger stepsizes are allowed, which in turn implies convergence for more negative values of $\rho$, thus capturing a larger class of problems. %
In addition, we show that the linesearch eventually passes without triggering any backtrack if  initialized based on the Jacobian of $F$ (cf. \cref{sec:curvature}).

\item We present a non-adaptive variant of our algorithm \eqref{eq:iter:constant},  and show that for particular parameter choices \eqref{eq:eg+} of \citet{diakonikolas2021efficient}, and when $\rho=0$  the celebrated forward-backward-forward (FBF) algorithm of \citet{Tseng2000modified} are  recovered, thus unifying and generalizing both methods. 
We improve upon \citet{diakonikolas2021efficient} by not only relaxing the problem class but also the stepsize range. 
We show that our results are tight by providing a matching lower bound, thus providing a complete picture of \eqref{eq:eg+} under weak MVI.

\end{enumerate}

\begin{figure}
\centering
  \includegraphics[width=0.35\textwidth]{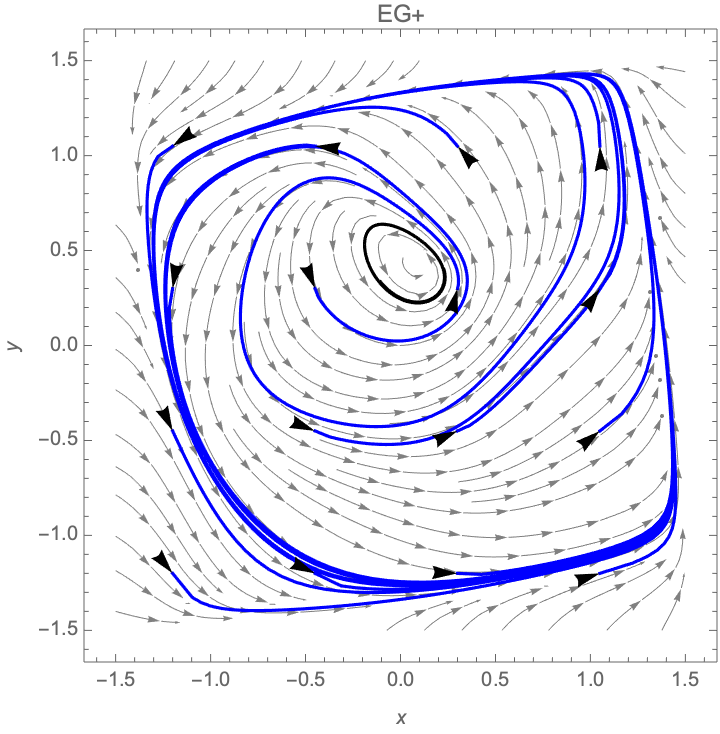}\qquad
  \includegraphics[width=0.35\textwidth]{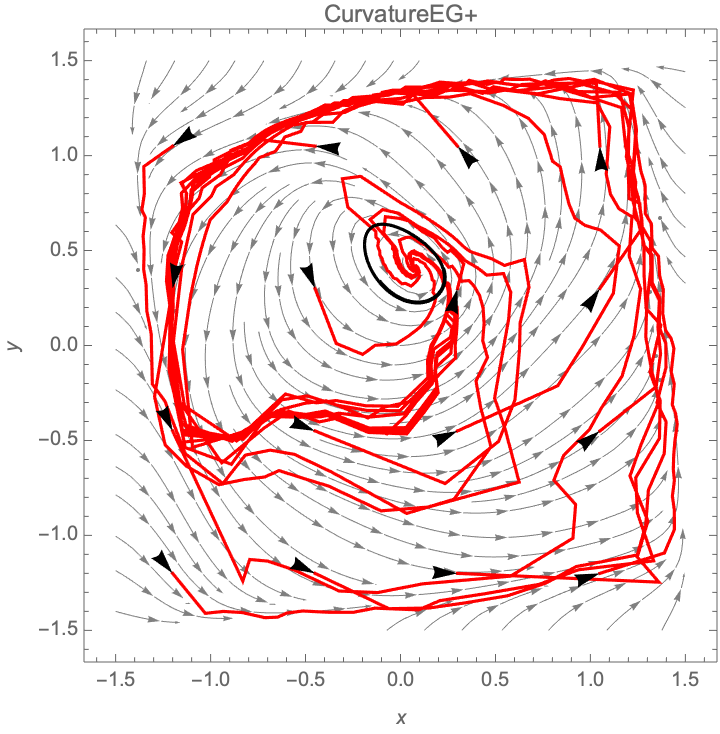}
  \caption{Forsaken \citep[Example 5.2]{hsieh2021limits} provides an example where the weak MVI constant $\rho$ does not satisfy algorithmic requirements of \eqref{eq:eg+}
  and \eqref{eq:eg+} does not converge to a stationary point but rather the attracting limit cycle (left). 
  In contrast, adaptively choosing the extrapolation stepsize large enough with our new method, called \eqref{eq:CurvatureEG}, is sufficient for avoiding the limit cycles (right).
  The repellant limit cycle is indicated in black and the stream plot shows the vectorfield $Fz$.
  \rebuttal{The blue and red curves indicate multiple trajectories of the algorithms starting from initializations indicated in black.
  See \Cref{app:Forsaken} for properties of Forsaken.}
  }
  \label{fig:forsaken}
\end{figure}

\vspace{-4mm}
\paragraph{Related work}

The community has resorted to various approaches to make progress for nonconvex-nonconcave minimax problems.
One line of work focuses on deriving local convergence results \citep{mazumdarFindingLocalNash2019,fiez2020gradient,heusel2017gans}.
For global results, the two primary approaches have been to either assume a global oracle for the inner problem 
\citep{jinWhatLocalOptimality2019,davisStochasticSubgradientMethod2018} or 
assume particular problem structure such as the Polyak-Łojasiewicz condition \citep{nouiehed2019solving,yang2020global} or concavity for the inner problem \citep{rafiqueNonConvexMinMaxOptimization2019}.

We follow the same tradition of assuming structure, but from the general perspective of operator theory.
The idea of studying minimax and related problems through the lens of variational inequality has a long history \citep{minty1962monotone,rockafellar1976monotone,polyak1987introduction,bertsekas1997nonlinear}, with recent renewed interest due to its relevance for minimax formulations \citep{mertikopoulos2018optimistic,gidel2018variational,azizian2020tight}.

One relaxation of the monotone case for which we have positive results is that of Minty variational inequalities (MVI) \citep{mertikopoulos2018optimistic,song2021optimistic,NIPS2017_e6ba70fc,liu2021first}, which includes all quasiconvex-concave and starconvex-concave problems.
\citet{diakonikolas2021efficient} introduced the relaxed condition of weak MVI.
In the unconstrained setting they showed non-asymptotic convergence results under a restricted problem constant $\rho$.
Similarly to us, \citet{lee2021fast} extends the regime but under the stronger condition of cohypomonotonicity.
They do so by studying a more evolved variant of extragradient building on anchoring techniques.
We instead directly improve upon \eqref{eq:eg+} and generalize it to new settings. %

In the stochastic setting, usually the stepsize for the extrapolation step is diminishing.
This is the case in \citet{bohm2020two} where they consider a forward-backward-forward type scheme.
However, they remain in the monotone setting, where the limit cycles are non-attracting, as exemplified by a bilinear game.
\citet{hsieh2021limits} recently showed that a large family of algorithms, which includes the extragradient method with diminishing stepsize, can converge to attracting limit cycles.
Going beyond this restriction, prior to \citet{diakonikolas2021efficient}, \citet{hsieh2020explore} interestingly considers two separate and diminishing stepsizes under the stronger assumption of MVI.

\section{Problem formulation and preliminaries}
\label{sec:setup}

In this paper we are interested in finding zeros of an operator (or set-valued mapping) $T:\R^n\rightrightarrows\R^n$ that is written as the sum of a Lipschitz continuous (but possibly nonmonotone)  operator $F$ and a maximally monotone operator $A$. That is, we wish to find $z\in \R^n$ such that the general inclusion
\begin{equation}\label{eq:StrucIncl}
	0\in Tz \coloneqq Az+ Fz
\end{equation}
holds. The set of all such points is denoted by $\zer T \coloneqq \set{z\in R^n}[0\in Tz]$.  
Throughout the paper problem \eqref{eq:StrucIncl} is studied under the following assumptions (definitions can be found in \Cref{sec:auxiliary}). 
\begin{ass}\label{ass:basic:Struct}
	In problem \eqref{eq:StrucIncl}, 
	\begin{enumerate}
	\item\label{ass:A:Struct}
		Operator $A:\R^n\rightrightarrows\R^n$ is a maximally monotone operator.  
	\item\label{ass:F:Struct}
		Operator $F:\R^n\to \R^n$ is $L$-Lipschitz continuous. %
	\item\label{ass:Minty:Struct} Weak Minty variational inequality (MVI) holds, \ie, there exists a nonempty set $\mathcal S^{\star}\subseteq \zer T$ such that for all $z^\star\in \mathcal S^{\star}$ and some $\rho\in(-\tfrac{1}{2L}, \infty)$
\begin{equation}
		\langle v,z - z^{\star}\rangle \geq \rho\|v\|^2, \quad \text{for all $(z, v)\in \graph T$.}
	\end{equation}
	\end{enumerate}
\end{ass}
Generally, we do not require the weak Minty assumption to hold at every $z^\star\in \zer T$. In fact, as shown in \cref{thm:main:Struct} 
nonemptiness of $\mathcal S^\star$ is sufficient for  ensuring  that the limit points belong to $\zer T$. Interestingly, despite nonmonotonicity of $F$, global (as opposed to subsequential) convergence can be established when $\mathcal S^\star = \zer T$, an assumption that is still weaker than cohypomonotonicity. 

VIs provide a convenient abstraction for a range of problems.
We mention some central examples below but otherwise defer to the overview in \citet{facchinei2007finite}.
Subsequently, we provide examples where the weak MVI holds.

\begin{example}[(minimax optimization).] \label{ex:minimax}
A comprehensive way to capture a wide range of applications in machine learning is to consider structured minimax problems of the form 
\begin{equation}
	\minimize_{x\in\R^{n_x}}\maximize_{y\in\R^{n_y}}\;  \mathcal L(x,y) \coloneqq \varphi(x,y) + g(x) -h(y),
\end{equation}
where $\varphi$ is not necessarily convex in $x$ or concave in $y$. Functions $g$ and $h$ are  proper extended real-valued lower semicontinuous and convex, with easy to compute proximal maps. 
 Common examples for $g$ and $h$ involve regularizers such as $\ell_1$, $\ell_2$ norms, or indicator functions of sets allowing us to capture constrained minimax problems. The first order optimality condition associated with this problem may be written in the form of the structured inclusion \eqref{eq:StrucIncl} by letting $Fz=(\nabla_x \varphi(x,y), -\nabla_y \varphi(x,y))$, $Az=(\partial g(x),\partial h(y))$.  
\end{example}

As it will become clear in the next section (cf. \Cref{alg:WeakMinty:Struct}), the main computations involved in the proposed scheme are evaluations of $F$ and resolvent $J_A = \left( \id + A\right)^{-1}$. Recall that the resolvent of a maximally monotone operator is firmly nonexpansive with full domain (cf. \cite[Sect. 23]{Bauschke2017Convex}). If $A=\partial f$ is the subdifferential operator of a convex function $f$,  then its resolvent is the proximal mapping. For instance when $A$ is as in \cref{ex:minimax}, then its resolvent is given by $J_A(x,y)= (\prox_{g}(x), \prox_h(y))$.

\begin{example}[($N$-player games).]\label{ex:game}
More generally, we can consider a continuous game of $N$ players in normal form.
Denote the decision variables $\bm z:=(z_i;z_{-i}):=(z_1,...,z_N)$ and let the loss incurred by the $i^{\mathrm{th}}$ player be $\mathcal L_i(z_i;z_{-i})=\varphi_i(\bm z)+g_i(z_i)$ where $\varphi_i$ is the payoff function and $g_i$ typically enforce constraints on $z_i$.
Then we seek a Nash equilibrium, which is any decision which is unilaterally stable, i.e.,
\begin{equation}
\mathcal L_i(z_i^\star;z_{-i}^\star) \leq \mathcal L_i(z_i;z_{-i}^\star) \quad \forall z_i\text{ and }i \in [N]\coloneqq \{1,\ldots,N\}.
\end{equation}
The corresponding first order optimality conditions may be written as $Az=(\partial g_1(z_1),\ldots, \partial g_N(z_N))$ and $Fz=(\nabla_{z_1}\varphi_1(\bm z),\ldots,\nabla_{z_N}\varphi_N(\bm z))$.
\end{example}

A solution to \eqref{eq:StrucIncl} thus returns a candidate for which the first order condition of the above problems is satisfied. 
In the monotone case these two solution concepts coincide, while in the more general case of weak MVI, we provide examples where this still holds.
In particular, we introduce in \cref{sec:toy-examples} a nonconvex-nonconcave minimax game which additionally exhibits limit cycles for $Fz$.
As a consequence most schemes including gradient descent ascent, extragradient and optimistic gradient descent ascent do not converge to a stationary point globally \citep{hsieh2021limits}.
However, the global Nash equilibrium satisfies \Cref{ass:Minty:Struct} with $\rho > -\nicefrac{1}{2L}$, which we show is sufficient for global convergence of \eqref{eq:iter:constant}.

The weak MVI condition is satisfied in certain reinforcement learning settings.
Specifically, \citet{diakonikolas2021efficient,daskalakis2021independent} considers a two-player zero-sum game where the weak MVI holds, while neither MVI nor cohypomonotonicity holds.
Interestingly, the formulation requires constraint---a condition they do not handle.
We thus provide the first provable algorithm for this setting.
Weak MVI also contains all quasiconvex-concave and starconvex-concave problems.
For further examples, the literature on cohypomonotonicity \citep{bauschke2020generalized} is relevant since it implies weak MVI, see for instance \citet[Example 1]{lee2021semi}.

\section{Generalizing Extragradient+}\label{sec:Struct}
\label{sec:analysis}

Our starting point is the Extragradient+ \eqref{eq:eg+} algorithm of \citet{diakonikolas2021efficient} which is identical to extragradient \citep{korpelevich1976extragradient} except for the second stepsize being smaller.
They only treat the inclusion \eqref{eq:StrucIncl} when $A\equiv 0$, and in our notation require $\rho\in(-\nicefrac{1}{8L}, 0]$.
Specifically,
\begin{equation}
\label{eq:eg+}
\tag{EG+}
\bar{z}^k = z^k - \gamma_k Fz^k, \quad 
z^{k+1} = z^k - \bar{\alpha}_k \gamma_k F\bar{z}^k
\end{equation}
where they choose $\gamma_k=\nicefrac{1}{L}$ and $\bar{\alpha}_k=\nicefrac 12$ \citep[Thm. 3.2]{diakonikolas2021efficient}.

We generalize \eqref{eq:eg+} in \Cref{alg:WeakMinty:Struct} to take the operator $A$ into account---consequently we capture constraint and regularized problems as well. %
In addition, the scheme is adaptive in $\bar{\alpha}_k$. 
We will show that the weaker requirement of $\rho\in(-\nicefrac{1}{2L}, \infty)$ suffices even for the more general inclusion \eqref{eq:StrucIncl}.

The main convergence results of \Cref{alg:WeakMinty:Struct} are established in the next theorem. The proof is largely inspired by recent developments in operator splitting techniques in the framework of monotone inclusions \citep{Latafat2017Asymmetric,pontus2021nonlinear}. The key idea lies in interpreting each iteration of the algorithm as a projection onto a certain hyperplane, an interpretation that dates back to \cite{Solodov1996Modified,solodov1999hybrid}.
\begin{thm} \label{thm:main:Struct}
Suppose that \cref{ass:basic:Struct} holds, and let $\lambda_k \in (0,2)$, $\gamma_k\in\big(\lfloor -2\rho\rfloor_+ ,\nicefrac1{L}\big]$ where $\lfloor x\rfloor_+\coloneqq \max\{0,x\}$,  ${\delta_k\in(\nicefrac{-\gamma_k}2, \rho]}$, $\liminf_{k\to \infty} \lambda_k(2-\lambda_k)>0$, and $\liminf_{k\to \infty} (\delta_k + \nicefrac{\gamma_k}2)>0$. Consider the sequences $\seq{z^k}$, $\seq{\z^k}$ generated by \Cref{alg:WeakMinty:Struct}. Then for all $z^\star \in \mathcal{S}^\star$, 
\begin{equation}\label{eq:sublin:main}
	{\min_{k=0,1,\ldots,m} \tfrac{1}{\gamma_k^2}\|H\z^k-Hz^k\|^{2}
	} 
	{}\leq{}
\tfrac{1}{\kappa(m+1)}\|z^{0} - z^\star\|^2, 
\end{equation}
where \(\kappa = \liminf_{k\to \infty} \lambda_{k}(2-\lambda_{k})(\delta_{k}+ \nicefrac{\gamma_{k}}{2})^{2}\). Moreover, the following holds 
\begin{enumerate}
	\item \label{thm:main:Struct:limitpoint} $\seq{\z^k}$ is bounded and its limit points  belong to $\zer T$; 
	\item  \label{thm:main:Struct:conv} if in addition $\limsup_{k\to \infty} \gamma_k<\nicefrac1{L}$ and $\mathcal S^\star = \zer T$, then $\seq{z^k}$, $\seq{\z^k}$ both converge to some  $z^\star\in \zer T$. 
\end{enumerate}

\end{thm}
Note that whenever $\limsup_{k\to\infty}\gamma_k < \tfrac1{L}$, \cref{lem:H:SM} may be used to derive a similar inequality in terms of $\|\z^{k}-z^{k}\|$ by lower bounding $\|H\z^k - Hz^k\|$ in \eqref{eq:sublin:main}. We also remark that tighter rates may be obtained in the regime $\rho\geq 0$, however, this will not be pursued in this work.

\begin{algorithm}[tb]
    \caption{(AdaptiveEG+) Deterministic algorithm for problem \eqref{eq:StrucIncl}}%
    \label{alg:WeakMinty:Struct}%
    \begin{algorithmic}[1]
\Initialize
	\(z^0 = z^{\rm init}\in\R^n\), $\lambda_k \in (0,2)$, $\gamma_k\in\big(\lfloor -2\rho\rfloor_+,\nicefrac1{L}\big]$, ${\delta_k\in(\nicefrac{-\gamma_k}2, \rho]}$, 
\item[\algfont{Repeat} for \(k=0,1,\ldots\) until convergence]

\State\label{state:barz:Struct}%
	Let 
	\(
		\bar z^k
	{}={}
		\big(\id + \gamma_k A\big)^{-1}\big(z^k - \gamma_k Fz^k \big)
	\) 
\State\label{State:alpha:Struct}%
Compute stepsize 
	\[
		\alpha_k 
	{}={}   
		\tfrac{\delta_k}{\gamma_k} + \frac{\langle \z^k - z^k,  H\z^k - Hz^k\rangle }{\|H\z^k - Hz^k\|^2}, 
	\]%
	\ \ \ where \(H=\id -\gamma_k F\).
\State\label{state:Finito:s+:Struct}%
Update the vector
	\(
		z^{k+1}
	{}={}
		z^k
		{}+{}
		\lambda_k \alpha_k (H\z^k - Hz^k)  
	\)
	
\item[\algfont{Return}]
	\(z^{k+1}\)
\end{algorithmic}

\end{algorithm}

    \subsection{Non-adaptive stepsize variant} 
    \label{sec:cnstStep}

Although we do not incur additional costs for evaluating the adaptive stepsize $\alpha_k$ in \cref{State:alpha:Struct}, it proves instructive to present a variant with constant stepsize. As a result we compare the range of our stepsizes against  \cite{diakonikolas2021efficient} showing an improvement by a factor of $\nicefrac32$. Moreover, in the monotone case ($\rho=0$), with a certain choice of stepsizes the algorithm reduces to the celebrated \emph{forward-backward-forward} (FBF) algorithm of \cite{Tseng2000modified}. We remark that the relation of FBF to projection-type algorithms was noted in \cite{Tseng2000modified}, \cite[Sect. 6.2.1]{pontus2021nonlinear}.

To this end, in this subsection consider the following non-adaptive variant of \Cref{alg:WeakMinty:Struct} that generalizes \eqref{eq:eg+}. Letting $\bar\alpha_k \in (0, 1 + \nicefrac{2\delta_k}{\gamma_k})$: 
\begin{equation}\tag{CEG+}\label{eq:iter:constant}
	\bar{z}^k 
		{}={}
	\big(\id + \gamma_k A\big)^{-1}\big(z^k - \gamma_k Fz^k\big), \quad 
	z^{k+1}
		{}={}
	z^k
		{}+{}
	\bar\alpha_k (H\z^k - Hz^k).%
\end{equation}
The convergence of this algorithm is an immediate byproduct of \cref{thm:main:Struct}. To see this, note that the $z^{k+1}$ update in \cref{state:Finito:s+:Struct} may be written as 
\(
	z^{k+1}
		{}={}
	z^k
		{}+{}
	2\eta_k\alpha_k (H\z^k - Hz^k)
\),
for $\eta_k\in(0,1)$. Therefore, convergence is still ensured for any $\bar\alpha_k < 2\alpha_k$ as the difference may be absorbed by the relaxation parameter $\eta_k$. 
Note that by $\nicefrac12$-cocoercivity of $H$ (cf. \cref{lem:H:coco})
\begin{equation}\label{eq:baralpha}
	\bar\alpha_k < \tfrac{2\delta_k}{\gamma_k} + 1
		 {}\leq{}   
	\tfrac{2\delta_k}{\gamma_k} + \frac{2\langle H\z^k - Hz^k, \z^k - z^k \rangle }{\|H\z^k - Hz^k\|^2} = 2\alpha_k,	
\end{equation}
establishing the validity of the prescribed stepsize range. The convergence of the non-adaptive variant is summarized in the next corollary that for simplicity is stated with constant parameters (dropping subscripts $k$). 

\begin{cor}[Constant stepsize] \label{cor:constant:Struct}
Suppose that \cref{ass:basic:Struct} holds, and let $\gamma\in\big(\lfloor-2\rho\rfloor_+, \nicefrac1{L}\big]$,  ${\delta\in(\nicefrac{-\gamma}2, \rho]}$, and $\bar\alpha \in (0, 1 + \nicefrac{2\delta}{\gamma})$. Consider the sequences $\seq{z^k}$, $\seq{\z^k}$ generated according to the update rule \eqref{eq:iter:constant}. Then, 
\begin{equation}\label{eq:sublinrate}
	{\min_{k=0,1,\ldots,m}\|H\z^k-Hz^k\|^{2}
	} 
	{}\leq{}
\frac{\|z^{0} - z^\star\|^2}{\kappa(m+1)}, 
\end{equation}
where $\kappa =\bar\alpha(1+\tfrac{2\delta}{\gamma} - \bar\alpha)$. Moreover, the claims of \cref{thm:main:Struct:limitpoint,thm:main:Struct:conv} hold true. 
\end{cor}
The setting of \cite{diakonikolas2021efficient} in \eqref{eq:eg+} involves the stepsizes $\gamma_k=\nicefrac{1}{L}$, $\alpha_k=\nicefrac12$. Note that when restricting to $A\equiv 0$, the iterates \eqref{eq:iter:constant} simplify to this form owing to the fact that  $H\z^k-Hz^k=\z^k - \gamma F\z^k-Hz^k = -\gamma F\z^k$. 
In comparison, in our setting if $\delta=\rho=-\nicefrac{1}{8L}$ (the smallest $\rho$ permitted in \cite{diakonikolas2021efficient}) is selected, then based on our analysis in \cref{cor:constant:Struct} we may select $\gamma_k  = \nicefrac{1}{L}$, and $\bar\alpha_k \in(0, \nicefrac34)$, thus the upper bound for the second stepsize is $\nicefrac32$ times that of \cite{diakonikolas2021efficient}.

\begin{rem}[relation to FBF]
	In \cref{cor:constant:Struct} the range of stepsizes $\gamma$, $\bar\alpha$ may alternatively be set as $\gamma\in\big(\lfloor-2\rho\rfloor_+, \nicefrac1{L}\big)$, $\bar\alpha \in (0, 1 + \nicefrac{2\delta}{\gamma}]$. This  is due to the fact that if $\gamma<\nicefrac1{L}$ (strictly), then $H$ is strictly $\nicefrac12$-cocoercive. Therefore, in \eqref{eq:baralpha}, $1 + \tfrac{2\delta}{\gamma} < 2\alpha_k$ holds, and thus the stepsize $\bar\alpha =1 + \tfrac{2\delta}{\gamma}$ is permitted.   
	Although this may appear to be of little practical significance, by setting $\gamma\in(0,\nicefrac{1}{L})$, $\delta = \rho=0$, and $\bar\alpha =1$ in \eqref{eq:iter:constant}, we obtain
	\(
	z^{k+1} = \z^k + \gamma Fz^k - \gamma F\z^k
	\), which is the forward-backward-forward (FBF) algorithm of \cite{Tseng2000modified}, \cite[Thm. 26.17]{Bauschke2017Convex}). 
\end{rem}

    \subsection{Lower bounds}
    \label{sec:lowerbounds}

We show that the result in \cref{cor:constant:Struct} is tight by providing a matching lower bound when $A\equiv 0$. 
We do so by fixing $\bar{\alpha}_k$ and showing a stepsize dependent lower bound.
In particular, note that if $\bar{\alpha}_k=\nicefrac{1}{2}$ as in \citet[Thm. 3.2]{diakonikolas2021efficient}, then \Cref{thm:lowerbound} implies a lower bound of $\rho > -\nicefrac{1}{4L}$ for the \eqref{eq:eg+} scheme. %
The lower bound is contextualized in \cref{fig:lowerbound} by relating it to our convergence results and existing results in the literature.

\begin{thm}\label{thm:lowerbound}
Consider a sequence $(z^k)_{k \in \mathbb N}$ generated according to \eqref{eq:eg+} fixing $\gamma_k = \gamma =\nicefrac{1}{L}$ and $\bar{\alpha}_k=\bar{\alpha}\in(0,1)$.
Let $-\rho L\geq\frac{1-\bar{\alpha} }{2}$.
Then, there exists an $F:\R^n\to \R^n$, $n>1$, satisfying \cref{ass:F:Struct} and \cref{ass:Minty:Struct} for which the sequence will not converge.
\end{thm}

\begin{figure}
\centering
\includegraphics[width=0.4\textwidth]{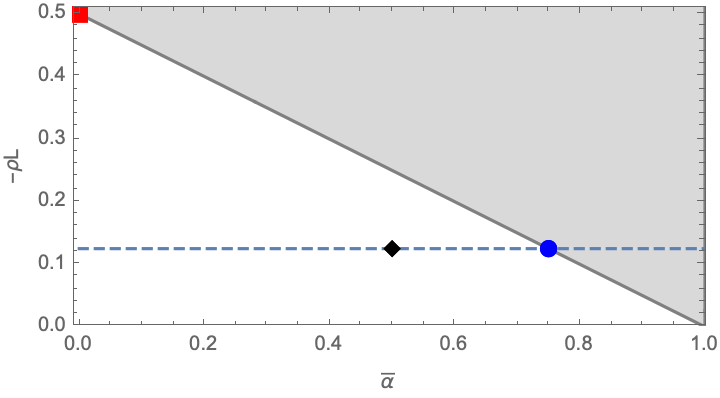}
\caption{
The grey region indicates where convergence provably cannot be guaranteed by \cref{thm:lowerbound}.
The dashed line indicates where $\rho=-\nicefrac 1{8L}$.
This is the condition under which \citep[Thm. 3.2]{diakonikolas2021efficient} shows the first convergence result $(\textcolor{black}{\blackdiamond})$.
\Cref{cor:constant:Struct} improves their result by matching the lower bound for any $\bar{\alpha}$, in particular for $\bar{\alpha}=\nicefrac{3}{4}$
$(\textcolor{blue}{\CIRCLE})$.
The adaptive scheme in \cref{thm:main:Struct} matches the smallest possible $\rho$ for any \eqref{eq:eg+} scheme with fixed stepsize $(\textcolor{red}{\blacksquare})$.
}
\label{fig:lowerbound}
\end{figure}

\section{Adaptively taking larger stepsizes using local curvature}
\label{sec:curvature}

\begin{algorithm}[b]
    \caption{Lipschitz constant backtracking}%
    \label{alg:LSBacktrack}%
    \begin{algorithmic}[1]
\Initialize
	\(z^k \in\R^n\), $\tau\in(0,1)$, $\nu\in(0,1)$
\State\label{state:BackTrack:Struct} 
\parbox[t]{0.98\linewidth}{%
Set initial guess $\gamma= \gamma^{\rm init}$, and let $\G_\gamma(z^k) \coloneqq \big(\id + \gamma A\big)^{-1}\big(z^k - \gamma Fz^k \big)$
\Statex{\bf while}
{
\(
	\gamma\|F(\G_\gamma(z^k)) - Fz^k\| >  \nu\|\G_\gamma(z^k) - z^k\|
\)}
{\bf do}
{
\(\gamma \gets \tau\gamma\)
}
}
	
\item[\algfont{Return}]
	\(\gamma_k = \gamma\) and \(\z^k = \G_\gamma(z^k)\)
\end{algorithmic}

\end{algorithm}

As made apparent in the analysis in \cref{sec:analysis}  (cf. \cref{proof:thm:main:Struct}) the bound on the smallest weak MVI constant $\rho$ in \cref{ass:Minty:Struct}  may be replaced with the requirement that $\rho>-\nicefrac{\gamma_k}2$ for all $k\in\N$. Therefore, larger stepsizes $\gamma_k$ would guarantee global convergence for an even larger class of problems.
Since a \emph{global} Lipschitz constant is inherently pessimistic the natural question then becomes how to locally choose a maximal stepsize without diverging. %

The proposed scheme involves a backtracking linesearch that uses the local curvature for its initial guess. 
The reason being that this will immediately pass, close enough to the solution $z^\star$, by argument of continuity.
More precisely, we will set the initial guess to something slightly smaller than $\|JF(z^k)\|^{-1}$,
where $JF(z)$ denotes the Jacobian of $F$ at $z$ and $\|\cdot\|$ is the spectral norm.
Note that, despite the use of second order information, the scheme remains efficient since $\|\jac F{z}\|$ only requires one eigenvalue computation performed through Jacobian-vector product \citep{pearlmutter1994fast}. %

Given an initial point $z^0 = z^{\rm init}$ and $\nu\in (0,1)$, the final scheme which we denote \eqref{eq:CurvatureEG} proceeds for $k=0,1,\dots$ as follows: %
\begin{equation}
\label{eq:CurvatureEG}
\tag{CurvatureEG+}
\begin{split}
1.\ & \text{Obtain $\gamma_k$ and $\z^k$ according to \Cref{alg:LSBacktrack} with $\gamma^{\mathrm{init}} = \nu \|JF(z^k)\|^{-1}$} \\
2.\ & \text{Compute $z^{k+1}$ according to \cref{State:alpha:Struct,state:Finito:s+:Struct} of \Cref{alg:WeakMinty:Struct} }%
\end{split}
\end{equation}

The above intuitive reasoning is made precise in the next lemma where it is shown that backtracking linesearch will terminate in finite time and that $\gamma^{\mathrm{init}}$ will be immediately accepted asymptotically.

\begin{lem}[Lipschitz constant backtracking]\label{lem:LS}
	Suppose that $F:\R^n\to\R^n$ is a $L$-Lipschitz continuous operator. Consider the linesearch procedure in \Cref{alg:LSBacktrack}. Then, 
	\begin{enumerate}
		\item \label{lem:LS:lb} The linesearch terminates in finite time with $\gamma \geq \min\{\gamma^{\rm init},\nicefrac{\nu\tau}{L}\}$; %
		\item\label{lem:LS:acc} Suppose that $\seq{z^k}$ converges to $z^\star\in \zer T$. If $F$ is continuously differentiable, and $\gamma^{\rm init} \in (0,\nu\|\jac F{z^k}\|^{-1})$ with $\nu\in(0,1)$, then eventually the backtrack will never be invoked ($\gamma^{\rm init}$ would be accepted). 
	\end{enumerate}

\end{lem}

The convergence results for \eqref{eq:CurvatureEG} are deduced based of the above lemma and \cref{thm:main:Struct} and are provided in \cref{cor:curvatureEG} in \cref{sec:proof:curvature}. 
We illustrate the behavior of \eqref{eq:CurvatureEG} in \cref{fig:forsaken} and in \cref{sec:experiments}. %

\section{Constructing toy examples}
\label{sec:toy-examples}

When \cref{ass:Minty:Struct} holds for negative $\rho$, limit cycles of the underlying operator $Fz$ can emerge.
We illustrate this with simple polynomial examples for which all the properties of interest can be computed in closed form.

\begin{definition}[PolarGame]
\label{def:polargame}
A PolarGame denotes a two-player game whose associated operator $F$ has limit cycles at $\|z\|_2=c_i$ for all $i\in[k]$ where $c_i\neq 0$.
\end{definition}
This turns out to be particularly easy to construct in polar coordinates as the name suggests (see \cref{app:polargame}).
Apart from introducing arbitrary number of limit cycles it also gives us control over $\rho$.
This is illustrated in the following instantiations capturing three important cases.%
\begin{example}[(PolarGame).]
\label{ex:polargame}
Consider $Fz = \left(\psi(x,y)-y, \psi(y,x)+x\right)$
where $\|z\|_\infty \leq \nicefrac{11}{10}$
and $\psi(x,y)=\frac{1}{16} a x (-1 + x^2 + y^2) (-9 + 16 x^2 + 16 y^2)$.
We have the following three cases:

\textit{(i)}
\ 
  $a=1$ then
  $\rho \in ( -\frac{1}{L},-\frac{1}{2L} )$
\ 
\textit{(ii)}
\ 
  $a=\frac{3}{4}$ then
  $\rho \in (-\frac{1}{2L}, -\frac{1}{3L})$
\ 
\textit{(iii)}
\ 
  $a=\frac{1}{3}$ then
  $\rho \in (-\frac{1}{8L}, -\frac{1}{10L})$

\rebuttal{where $L$ denotes the Lipschitz constant of $F$ restricted to the constraint set.}
For all cases $F$ exhibits limit cycles at $\|z\|=1$ and $\|z\|=\nicefrac{3}{4}$.
Proof is deferred to \cref{app:polargame-props}.
\end{example}

\begin{example}[(minimax).]
\label{ex:globalforsaken}
In the particular case of constrained minimax problem we introduce the following polynomial game:
\begin{equation}
\label{eq:globalforsaken}
\tag{GlobalForsaken}
\minimize_{|x|\leq\nicefrac{4}{3}} \maximize_{|y|\leq\nicefrac{4}{3}} \phi(x,y):=xy+\psi(x)-\psi(y),
\end{equation}
where $\psi(z) = \frac{2 z^6}{21}-\frac{z^4}{3}+\frac{z^2}{3}$. 
We provide proof of the following properties in \cref{app:GlobalForsaken}: 
\begin{enumerate}
\item There exists a repellant limit cycle and an attracting limit cycle of $F$.
\item \rebuttal{$z^\star=(0,0)$} is a global Nash equilibrium for which \cref{ass:Minty:Struct} holds \rebuttal{inside the constraint} with $\rho > -\nicefrac{1}{2L}$, \rebuttal{where $L$ denotes the Lipschitz constant of $F$ restricted to the constraint set.}
\end{enumerate}
\end{example}

\section{Experiments}
\label{sec:experiments}

The algorithms considered in the experiments include the adaptive \cref{alg:WeakMinty:Struct}, \eqref{eq:CurvatureEG}, and constant stepsize methods that can be seen as instances of \eqref{eq:iter:constant} for various choices of $\gamma_k$ and $\bar{\alpha}_k$.
When $\gamma_k=\nicefrac 1L$ and $\bar{\alpha}_k=1$ we recover a constrained variant of extragradient, which we denote CEG. 
When $\bar{\alpha}_k=\nicefrac{1}{2}$ we denote the scheme CEG+, which is the direct generalization to the constraint setting of the \eqref{eq:eg+} scheme studied in \citet[Thm. 3.2]{diakonikolas2021efficient}.
Note that this choice of $\bar{\alpha}_k$ restricts the problem class for which we otherwise can have guaranteed convergence according to \Cref{cor:constant:Struct}.
When $\bar{\alpha}_k$ is chosen adaptively according to \Cref{alg:WeakMinty:Struct} we refer to it as AdaptiveEG+. Finally, when $\gamma_k$ is additionally chosen adaptively we use the name \eqref{eq:CurvatureEG}.

In the stochastic setting, when $\gamma_k=\nicefrac{1}{k}$ and $\alpha_k=1$, effectively both stepsizes diminish, and we recover a constrained variant of the popular stochastic extragradient scheme (see e.g. \citet[Algorithm 3]{hsieh2021limits}), which we refer to as SEG.
We also consider a heuristic variant where $\gamma_k=\nicefrac{1}{L}$ and only $\alpha_k$ is decreasing, which we refer to as SEG+.

We test the algorithms on the constructed examples and confirm their convergence guarantees.
Specifically, we apply the algorithms to the minimax problem in \cref{ex:globalforsaken}, the PolarGames in \cref{ex:polargame}, and a worst case construction, \cref{ex:lowerbound}, from the proof of the lower bound (cf. \Cref{sec:proofs:lowerbounds}).
For \cref{ex:lowerbound} we choose the problem parameters such that $\rho = -\nicefrac{1}{3L}$ according to \eqref{eq:lowerbound-worstcase}, and additionally add an $\ell_\infty$-ball constraint to keep the iterates bounded.
To simulate the stochastic setting we add Gaussian noise to calls of $F$.
Results for the deterministic setting and stochastic setting can be found in \cref{fig:det} and \cref{fig:stoc} respectively.

\begin{figure}[tb]
\centering
\begin{subfigure}{0.38\textwidth}
\caption{\label{fig:det-a}\Cref{ex:polargame} ($a=1$)}
\includegraphics[width=\textwidth]{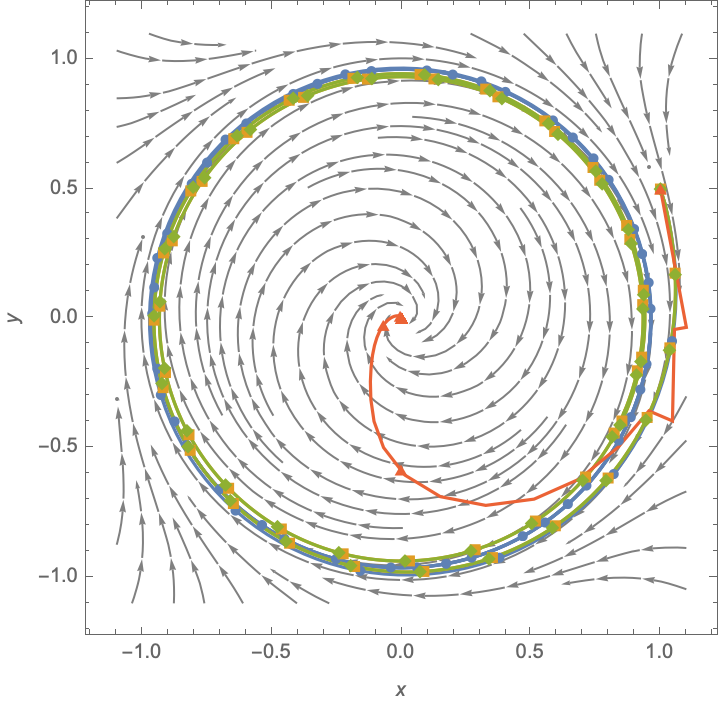}
\end{subfigure}\qquad
\begin{subfigure}{0.38\textwidth}
\caption{\label{fig:det-b}\Cref{ex:lowerbound} ($\rho = -\nicefrac{1}{3L}$)}
\includegraphics[width=\textwidth]{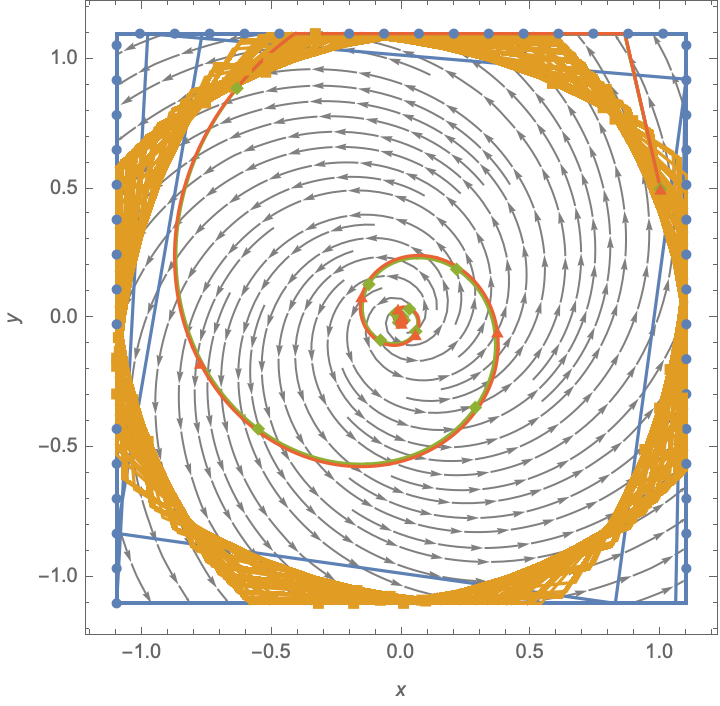}
\end{subfigure}
\includegraphics[width=0.4\textwidth]{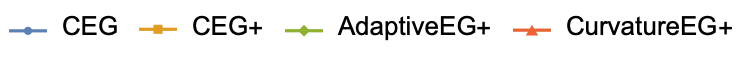}
\caption{
  Deterministic setting.
  In (\subref{fig:det-a}) we have an instance of \Cref{ex:polargame} with
  $\rho < -\nicefrac{1}{2L}$ for which \Cref{thm:lowerbound} provides lower bound for extrapolation stepsize $\gamma_k=\nicefrac{1}{L}$.
  However, adaptively choosing $\gamma_k$ larger can converge as illustrated with \eqref{eq:CurvatureEG}.
  In addition, (\subref{fig:det-b}) confirms with \Cref{ex:lowerbound}, that \eqref{eq:iter:constant} for $\bar{\alpha}_k=\nicefrac{1}{2}$ and CEG may indeed not converge even when $\rho = -\nicefrac{1}{3L}$.
  In contrast, both AdaptiveEG+ and \eqref{eq:CurvatureEG} converges to the stationary point. 
  Note that picking $\bar{\alpha}_k < \nicefrac{1}{3}$ would lead to convergence of \eqref{eq:iter:constant} by \Cref{cor:constant:Struct}.
  See \cref{fig:polargame} and \cref{fig:globalforsaken-cegplus} for supplementary experiments.
}
\label{fig:det}
\end{figure}

\begin{figure}[tb]
\centering
\begin{subfigure}{0.38\textwidth}
\caption{\label{fig:stoc-a}\Cref{ex:globalforsaken}}
\includegraphics[width=\textwidth]{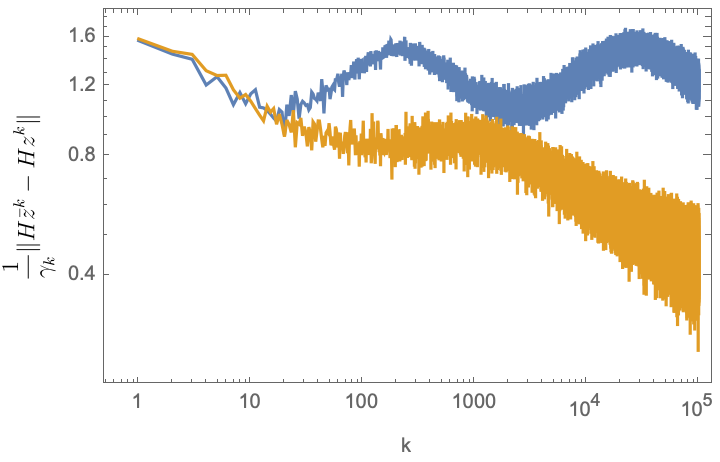}
\end{subfigure}\qquad
\begin{subfigure}{0.38\textwidth}
\caption{\label{fig:stoc-b}\Cref{ex:lowerbound}  ($\rho = -\nicefrac{1}{3L}$)}
\includegraphics[width=\textwidth]{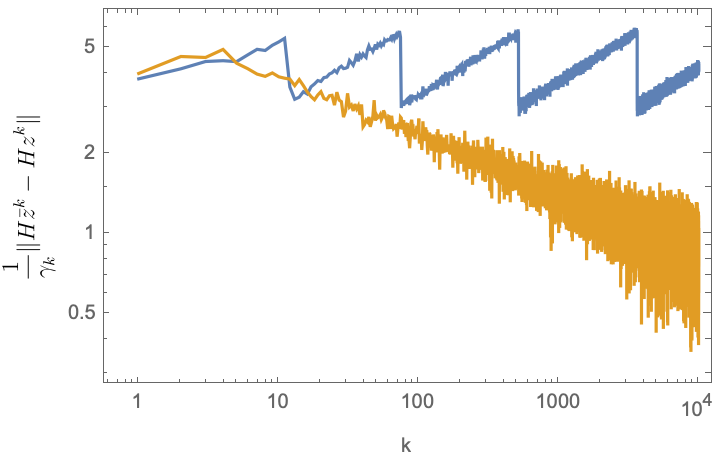}
\end{subfigure}
\includegraphics[width=0.4\textwidth]{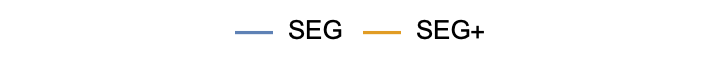}
\caption{
  Stochastic setting.
  In (\subref{fig:stoc-a}) we test the stochastic algorithms on our nonconvex-nonconcave constrained minimax example.
  The cycling behavior of SEG is inline with \citet{hsieh2021limits}, who shows that the sequence generated by SEG can converge to limit cycles of the underlying operator $F$.
  On the other hand, we observe that SEG+ escapes the attracting limit cycle.
  In (\subref{fig:stoc-b}) we also provide a more challenging example motivated by our lower bound.
}
\label{fig:stoc}
\end{figure}

\section{Conclusion}
\label{sec:conclusion}

This paper introduced an EG-type algorithm for a class of nonconvex-nonconcave minimax problems that satisfy the \emph{weak Minty variational inequality} (MVI). The range of parameter in the weak MVI was extended compared to EG+ of \cite{diakonikolas2021efficient}, and tightness of our results were demonstrated through construction of a counter example. In addition, EG+ \citep{diakonikolas2021efficient}, as well as the forward-backward-forward algorithm \citep{Tseng2000modified} were all shown to be special cases of our scheme. Furthermore, \eqref{eq:CurvatureEG} was proposed that performs a backtracking linesearch on the extrapolation stepsize $\gamma_k$
allowing for larger stepsizes and relaxes the condition $\rho>\tfrac{1}{2L}$ to $\rho >\nicefrac{-\gamma_k}{2}$ which is often a much weaker condition. More importantly, it is shown that asymptotically the linesearch always passes with $\gamma_k= \nu\|\jac F{z^k}\|^{-1}$ for any $\nu\in(0,1)$, thus ratifying the name \eqref{eq:CurvatureEG}. %
 Future direction include exploring applications of the proposed algorithm in particular in the setting of GANs. It is also interesting to develope a variance reduced variant of the algorithm for finite sum minimax problems.

\section{Acknowledgments and disclosure of funding}
\label{sec:acknowledgement}
We would like to especially thank Yu-Guan Hsieh for providing valuable feedback and discussion.
This project has received funding from the European Research Council (ERC) under the European Union's Horizon 2020 research and innovation programme (grant agreement n° 725594 - time-data).
This work was supported by the Swiss National Science Foundation (SNSF) under  grant number 200021\_205011.
The work of the second and third author was supported by the Research
Foundation Flanders (FWO)
postdoctoral grant 12Y7622N and research projects G081222N, G0A0920N,
G086518N,
and G086318N; Research Council KU Leuven C1 project No. C14/18/068;
Fonds de la Recherche Scientifique -- FNRS and the Fonds
Wetenschappelijk Onderzoek -- Vlaanderen under EOS project no 30468160
(SeLMA); European Union's Horizon 2020 research and innovation
programme under the Marie Skłodowska-Curie grant agreement No. 953348.
The work of Olivier Fercoq was supported by the Agence National de la Recherche grant ANR-20-CE40-0027, Optimal Primal-Dual Algorithms (APDO).

\bibliography{TeX/WMVI_ICLR2022,TeX/Bibliography}
\bibliographystyle{iclr2022_conference.bst}

\newpage
\appendix

\section{Preliminary definitions}\label{sec:auxiliary}
    
Notationally we will use $\lfloor x\rfloor_+\coloneqq \max\{0,x\}$ throughout.
We additionally recall some standard definitions and results and refer to  
\cite{Bauschke2017Convex,Rockafellar1970Convex}) for further details.

An operator or set-valued mapping $A:\R^n\rightrightarrows\R^d$ maps each point $x\in\R^n$ to a subset $Ax$ of $\R^d$. We will use the notation $A(x)$ and $Ax$ interchangably. 
 We denote the domain of $A$ by $$\dom A\coloneqq\{x\in\R^n\mid Ax\neq\emptyset\},$$
its graph by $$\graph A\coloneqq\{(x,y)\in\R^n\times \R^d\mid y\in Ax\},$$ and 
the set of its zeros by $\zer A\coloneqq\{x\in\R^n \mid 0\in Ax\}$. 
The inverse of $A$ is defined through its graph: $\graph A^{-1}\coloneqq\{(y,x)\mid (x,y)\in\graph A\}$.
The \emph{resolvent} of $A$ is defined by $J_A\coloneqq(\id+A)^{-1}$, where $\id$ denotes the identity operator.  

\begin{defin}[(co)monotonicity \cite{bauschke2020generalized}] An Operator $A:\R^n\rightrightarrows\R^n$ is said to be $\rho$-monotone for some $\rho\in \R$, if for all $(x,y),(x^\prime,y^\prime)\in\graph A$
 \[
 \rho\|x-x^\prime\|^2\leq \langle x-x^\prime,y-y^\prime\rangle, 
 \]
 and it is said to be $\rho$-comonotone if for all $(x,y),(x^\prime,y^\prime)\in\graph A$
\[
 \rho\|y-y^\prime\|^2\leq \langle x-x^\prime,y-y^\prime\rangle.
 \]
The operator $A$ is said to be maximally (co)monotone if its graph is not strictly contained in the graph of another (co)monotone operator.
\end{defin}
We say that $A$ is monotone if it is $0$-monotone. When $\rho<0$, $\rho$-comonotonicity is also referred to as $|\rho|$-cohypomonotonicity.
\begin{defin}[Lipschitz continuity and cocoercivity]
	Let $\mathcal D\subseteq \R^n$ be a nonempty subset of $\R^n$. A single-valued operator $A:\mathcal D\to \R^n$ is said to be $L$-Lipschitz continuous if for any $x,x^\prime\in \mathcal D$
	\[
	\|Ax - Ax^\prime\| \leq L\|x-x^\prime\|,
	\]
	and $\beta$-cocoercive if 
	\[
		\beta\|Ax - Ax^\prime\|^2 \leq \langle x-x^\prime, Ax - Ax^\prime \rangle. 
	\]
	Moreover, $A$ is said to be nonexpansive if it is $1$-Lipschitz continuous, and firmly nonexpansive if 
	it is $1$-cocoercive. 
\end{defin}
  The resolvent operator  $J_A$ is firmly nonexpansive (with $\dom J_A= \R^n$) if and only if $A$ is (maximally) monotone. 

The following lemma plays an important role in our convergence analysis. 
\begin{lem}\label{lem:H}
	Let  $A:\R^n\to \R^n$ denote a single valued operator. Then, 
	\begin{enumerate}
		\item\label{lem:H:coco} $A$ is $1$-Lipschitz if and only if $T=\id - A$ is $\nicefrac12$-cocoercive. 
		\item \label{lem:H:SM} If $A$ is $L$-Lipschitz, then $T=\id - \eta A$, $\eta\in(0,\nicefrac1{L})$, is $(1-\eta L)$-monotone, and 
		in particular 
		\(
		\|Tu-Tv\| \geq (1-\eta L)\|u-v\|
		\) for all $u,v\in\R^n$. 
	\end{enumerate}
\begin{proof}
	The first claim follows directly from \cite[Prop.4.11]{Bauschke2017Convex}. That $T$ is strongly monotone is a consequence of the Cauchy Schwarz inequality and Lipschitz continuity of $A$: 
	\begin{align*}
		\langle Tv - Tu, v-u \rangle  
			{}={} &
		\|v-u\|^2 - \eta\langle Av - Au, v-u\rangle  
			{}\geq{} 
		(1-\eta L)\|v-u\|^2.
	\end{align*}
	In turn, the last claim follows from the Cauchy-Schwarz inequality.
\end{proof}
\end{lem}

\section{Proofs and further results}
\subsection{Proofs of \texorpdfstring{\Cref{sec:Struct}}{\S\ref*{sec:Struct}}}\label{sec:proofs:Struct}

\begin{appendixproof}{thm:main:Struct}
	Let $H=\id -\gamma_k F$. 
 By \Cref{state:barz:Struct} 
\(
Hz^k  \in \z^k + \gamma_k A\z^k
\). Therefore,
\begin{equation}\label{eq:1}
	\tfrac{1}{ \gamma_k}(Hz^k - H\z^k) %
		{}\in{}
    A\z^k + F\z^k %
\end{equation}
In what follows we will show that \Cref{alg:WeakMinty:Struct} is equivalent to taking a forward-backward step followed by a correction step. Consider the updates
	\begin{align*}
	\z^k	 
		 {}\coloneqq{} &  
	\left(\id + \gamma_k A\right)^{-1}	 
	\big(z^k -  \gamma_k  Fz^k\big), 
	\numberthis\label{eq:barz}\\
	 z^{k+1} 
	 	 {}={} &
  	 (1-\lambda_k)z^k + \lambda_k\proj_{\D_k}(z^k), \quad \text{where}\quad  \D_k\coloneqq \set{w \mid  \langle  Hz^k - H\z^k, \z^k - w \rangle  \geq \tfrac{\delta_k}{\gamma_k}\|Hz^k - H\z^k\|^2}. %
\end{align*}
Note that 
\begin{align*}
	 \langle H\z^k - Hz^k, \z^k - z^k \rangle + \tfrac{\delta_k}{\gamma_k} \|H\z^k - Hz^k\|^2  
		{} \geq {} &
	(\tfrac{1}{2} + \tfrac{\delta_k}{\gamma_k}) \|H\z^k - Hz^k\|^2 
\numberthis\label{eq:posAlpha}
\end{align*}
where in the inequality \cref{lem:H:coco} was used. Hence, by \eqref{eq:posAlpha} the stepsize $\alpha_k$ is positive and bounded away from zero. Moreover, if $z^k\in \D_k$, then from \eqref{eq:posAlpha} we may conclude that $\|H\z^k -Hz^k\|\leq 0$ which implies that the generated sequence remains constant and $\z^k\in \zer T$ (cf. \eqref{eq:1}). 

The projection onto $\D_k$ for any $v\notin \D_k$ is given by 
\begin{equation*} %
\proj_{\D_k}(v)  = v + \frac{\langle \z^k - v, Hz^k - H\z^k\rangle -\tfrac{\delta_k}{\gamma_k}\|Hz^k - H\z^k\|^2}{\|Hz^k - H\z^k\|^2}(Hz^k - H\z^k)	
\end{equation*}
Moreover, \eqref{eq:1} together with \cref{ass:Minty:Struct} at $\bar z^k$ yields 
\begin{equation}\label{eq:mintyH}
	 \tfrac{1}{ \gamma_k}\langle  Hz^k - H\z^k, \z^k - z^\star \rangle  \geq \tfrac{\rho}{ \gamma_k^2}\|Hz^k - H\z^k\|^2\geq \tfrac{\delta_k}{ \gamma_k^2}\|Hz^k - H\z^k\|^2,  
\end{equation}
thus ensuring $z^\star\in \mathcal S^\star \subseteq \D_k$. The projection onto $\D_k$ is then given by 
\(
	\proj_{\D_k}(z^k) = z^k + \alpha_k (H\z^k-Hz^k),   
\)
where $\alpha_k$ is as in \cref{State:alpha:Struct}. %

Finally, since the projection $\proj_{\D_k}$ is firmly nonexpansive, it follows from \cite[Cor. 4.41]{Bauschke2017Convex} that the mapping $(1-\lambda_k)\id + \lambda_k \proj_{\D_k}$ is $\nicefrac{\lambda_k}2$-averaged. Consequently, %
we may conclude that $\seq{z^k}$ is Fej\'er monotone relative to $ \mathcal S^\star$ \cite[Prop. 4.35(iii)]{Bauschke2017Convex}. That is for all $z^\star\in  \mathcal S^\star$  
\begin{align*}
	\|z^{k+1} - z^\star\|^2 
		{}\leq{} &
	\|z^k - z^\star\|^2 - \lambda_k{(2-\lambda_k)}\alpha_k^2\|H\z^k-Hz^k\|^2.
	\\%
\dueto{\eqref{eq:posAlpha}}		{}\leq{} &
	\|z^{k}-z^{\star}\|^{2}-\tfrac{\varepsilon_k}{\gamma_{k}^2}\|H\z^k-Hz^k\|^{2}, \numberthis \label{eq:Fejer}
\end{align*}
where $\varepsilon_k \coloneqq \lambda_{k}(2-\lambda_{k})(\tfrac{\gamma_{k}}{2}+\delta_k)^{2}$. 
The convergence rate in \eqref{eq:sublin:main} is obtained by telescoping \eqref{eq:Fejer}.
Since $\liminf_{k\to \infty}\varepsilon_k>0$, $\seq{\tfrac{1}{\gamma_{k}^2}\|H\z^{k}-Hz^k\|^{2}}$ converges to zero. 
Moreover, $\seq{\|z^k-z^\star\|^2}$ converges and the sequence $\seq{z^k}$ is bounded. 
Since $\gamma_k$ is bounded, and $F$ and the resolvents $(\id + \gamma_k A)^{-1}$ are Lipschitz continuous (cf. \cite[Cor. 23.9]{Bauschke2017Convex}), so is their composition. 
 Hence, $\seq{\z^k}$ is also bounded. Let $\seq{\bar z^k}[k\in K]$ be a subsequence converging to some  $\z\in\R^n$.
  Combined with the fact that $\seq{\tfrac{1}{\gamma_{k}^2}\|H\z^{k}-Hz^k\|^{2}}$ converges to zero, we may conclude from \eqref{eq:1} along with \cite[Prop. 20.38]{Bauschke2017Convex} and Lipschitz continuity of $F$ that $\z\in \zer T$. Finally, if in addition $\gamma = \limsup_{k\to \infty} \gamma_k <1/L$, then $(1-\gamma L)\|\z^k - z^k\|\leq \|H\z^k-Hz^k\|$ (invoke \cref{lem:H:SM}). Therefore, $\seq{\|\z^k - z^k\|}$ converges to zero, which in turn implies that a subsequence $\seq{z^k}[k\in K']$ converges to a point $z'$ iff so does the subsequence $\seq{\z^k}[k\in K']$. Hence, $\seq{z^k}[k\in K]$ also converges to $\z \in \zer T$.  Consequently, if \cref{ass:Minty:Struct} holds at all of the zeros of $T$, \ie, if $\mathcal S^\star = \zer T$, then the second claim follows by invoking \cite[Thm. 5.5]{Bauschke2017Convex}. 
\end{appendixproof}

\begin{appendixproof}{cor:constant:Struct}
The proof of convergence was already given prior to the statement of the corollary. It remains to derive \eqref{eq:sublinrate}. By \cref{ass:Minty:Struct} and owing to $\nicefrac12$-cocoercivity of $H$ (cf. \cref{lem:H:coco})
\begin{align*}
	\langle z^{k}-  z^{\star},H {\z^k} - H {z^k}\rangle
		{}={}&
	\langle \bar{z}^{k}- z^{\star},H{\z^k} - H{z^k}\rangle
		{}+{}
	\langle z^{k}- {\z^k},H{\z^k} - H{z^k}\rangle
	\\%
	\dueto{\eqref{eq:mintyH}}
		{}\leq{}& 
	-(\tfrac{1}{2} + \tfrac{\delta_k}{\gamma_k})  \|H\z^k - Hz^k \|^2. 		
	\numberthis\label{eq:innprodlower}
\end{align*}

Therefore, provided that $\bar\alpha>0$ we have 
\begin{align*}
	\|z^{k+1} - z^\star\|^2 
		{}={} &
	\|z^k - z^\star\|^2 + \bar\alpha^2\|H\z^k-Hz^k\|^2 + 2\bar\alpha \langle z^k - z^\star, H\z^k-Hz^k \rangle
	\\%
\dueto{\eqref{eq:innprodlower}}
		{}\leq{} &
	\|z^{k}-z^{\star}\|^{2}-\bar\alpha(2(\tfrac{1}{2} + \tfrac{\delta}{\gamma}) - \bar\alpha)\|H\z^k-Hz^k\|^{2}.
\end{align*}
Telescoping the above inequality yields the claimed inequality. 
\end{appendixproof}

\subsection{Convergence results and proofs of \texorpdfstring{\Cref{sec:curvature}}{\S\ref*{sec:curvature}}}\label{sec:proof:curvature}

The convergence results for \eqref{eq:CurvatureEG} are provided in the next corollary where $\rho$ in \cref{ass:Minty:Struct} is allowed to take potentially larger values provided that $\rho>-\nicefrac{\gamma_k}2$. 
Note that owing to the lower bound on $\gamma_k$ (cf. \cref{lem:LS:lb}), the weak MVI assumption in the corollary is always satisfied if $\rho\in(-\nicefrac{\nu\tau}{2L},\infty)$, however, in practice $\gamma_k$ may take larger values.

\begin{cor}\label{cor:curvatureEG}
Suppose that \cref{ass:A:Struct,ass:F:Struct} hold, and consider the sequences $\seq{z^k}$, $\seq{\z^k}$ generated by \eqref{eq:CurvatureEG}. Suppose that \cref{ass:Minty:Struct} holds for some $\rho\in\R$ satisfying $\gamma_k+2\rho>0$, and let  ${\delta_k\in(\nicefrac{-\gamma_k}2, \rho]}$, $\lambda_k \in (0,2)$,
 $\liminf_{k\to \infty} \lambda_k(2-\lambda_k)>0$, and $\liminf_{k\to \infty} (\delta_k + \nicefrac{\gamma_k}2)>0$. Then, 
\begin{enumerate}
    \item \label{thm:Curv:Struct:vanish} The sequence $\seq{\|\z^k-z^k\|^2}$ vanishes;
    \item \label{thm:Curv:Struct:limitpoint} $\seq{\z^k}$, $\seq{z^k}$ are bounded, and have the same limit points belonging to $\zer T$; 
    \item  \label{thm:Curv:Struct:conv} if in addition $\mathcal S^\star = \zer T$, then $\seq{z^k}$, $\seq{\z^k}$ both converge to some  $z^\star\in \zer T$. 
\end{enumerate}
Moreover, if $z^k, \z^k\to z^\star\in \zer T$ (as is the case in \ref{thm:Curv:Struct:conv}), and $F$ is continuously differentiable, then 
 eventually the backtrack will never be invoked. %

\begin{proof}
 Observe that in the proof of \cref{thm:main:Struct} $1$-Lipschitz continuity of $\gamma_k F$ is only used at the generated points $\z^k$ and $z^k$ (see \eqref{eq:posAlpha}), and is thus ensured by the linesearch \cref{alg:LSBacktrack}. 
Therefore, it is easy to see that $\alpha_k$ is positive and bounded away from zero provided that $\rho> -\nicefrac{\gamma_k}2$ , see \eqref{eq:posAlpha}. 
Moreover, since $\gamma_k\|F\z^k - Fz^k\|\leq \nu\|\z^k - z^k\|$, arguing as in \cref{lem:H:SM} we obtain $\|H\z^k-Hz^k\|\geq (1-\nu)\|\z^k-z^k\|$. Hence, it follows from \eqref{eq:Fejer} that 
\begin{align*}
    \|z^{k+1} - z^\star\|^2 
        {}\leq{} &
    \|z^{k}-z^{\star}\|^{2}-\tfrac{\varepsilon_k(1-\nu)}{\gamma_{k}^2}\|\z^k-z^k\|^{2}, 
\end{align*}
By telescoping the inequality and noting that $\gamma_k$ is bounded, we obtain $\sum_{k\in \N}\|\z^k - z^k\|^2<\infty$, implying \ref{thm:Curv:Struct:vanish}. Noting this and arguing as in the last part of the proof of  \cref{thm:main:Struct} establishes \ref{thm:Curv:Struct:limitpoint}, \ref{thm:Curv:Struct:conv}. The last claim is the direct consequence of \cref{lem:LS:acc}. 
\end{proof}
\end{cor}

\begin{appendixproof}{lem:LS}
	\ref{lem:LS:lb}:	
		Since $F$ is $L$-Lipschitz continuous the linesearch would terminate in finite steps. Either $\gamma^{\rm init}$ satisfies the condition, or else the backtrack procedure is invoked, which in turn implies the previous candidate $\gamma/\tau$ should have violated the condition leading the the claimed lower bound.  

	\ref{lem:LS:acc}: Since the resolvent $(\id + \gamma A)^{-1}$ and $F$ are Lipschitz continous, so is their composition. Hence, $\G_\gamma(z^k) \to \G_\gamma(z^\star)$. Furthermore, by definition 
	\(
	z^\star - \gamma Fz^\star \in \G_\gamma(z^\star) + \gamma A(\G_\gamma(z^\star))
	\). 
	Consequently, using monotonicity of $A$ at $\G_\gamma(z^\star)$ and $z^\star$, and that $-Fz^\star\in Az^\star$ yields 
	\(
	0 
		{}\leq{}
	\langle z^\star - \gamma Fz^\star -\G_\gamma(z^\star) + Fz^\star, \G_\gamma(z^\star) - z^\star\rangle
			{}={}
	- \|z^\star -\G_\gamma(z^\star)\|^2
	\). Thus $\G_\gamma(z^\star)=z^\star$. Using the fact that both \(\seq{\G_\gamma(z^k)}\) and  \(\seq{z^k}\) converges to $z^\star\in \zer T$:
	\[
	\lim_{k\to \infty} \frac{\|F(\G_\gamma(z^k)) - Fz^k\|}{\|\G_\gamma(z^k)-z^k\|} 
		{}\leq{}
	\limsup_{z,z^\prime \to z^\star} \frac{\|Fz^{\prime} - Fz\|}{\|z^\prime - z\|} 
		{}={} 
	\lip F(z^\star)
		{}={} 
	\|\jac F{z^\star}\|,
	\]
	where \cite[Thm. 9.7]{rockafellar2009variational} was used. The claim follows from continuity of $\jac F{}$ and the fact that $\seq{z^k}$ converges to $z^\star$. 
\end{appendixproof}

\subsection{Proofs of \texorpdfstring{\Cref{sec:lowerbounds}}{\S\ref*{sec:lowerbounds}}}\label{sec:proofs:lowerbounds}
    
To prove the lower bound we introduce the following unconstrained bilinear minimax problem with an unstable critical point.

\begin{example}
\label{ex:lowerbound}
Consider the following minimax problem:
\begin{equation}
\label{eq:lowerbound-example}
\minimize_{x \in \mathbb R} \maximize_{y \in \mathbb R} f(x,y) := a x y+\frac{b}{2}(x^2-y^2),
\end{equation}
where $b<0$ and $a>0$.
\end{example}

\begin{appendixproof}{thm:lowerbound}
The associated operator of \cref{ex:lowerbound} can easily be computed,
\begin{equation}
Fz=\left( a y+b x, b y-a x
\right),
\end{equation}
where $z=(x,y)$.
In this particular case, both $L$ and $\rho$ turn out to be constants.
By simple calculation we have,
\begin{equation}
\|JF(z)\| = \sqrt{a^2 + b^2}, \quad 
\rho = \frac{b}{a^2+b^2}
\end{equation}
where $\|\cdot\|$ is the spectral norm.
Since the norm of the Jacobian is constant it equates the global Lipschitz constant, $L=\|JF(z)\|$.

By linearity of $F$, one step of \eqref{eq:eg+} is conveniently also a linear operator.
Specifically,
\begin{equation}
\label{eq:bilinear-operator}
z^{k+1}= Tz^{k} \quad \text{with} \quad T:=\left(
\begin{array}{cc}
 \frac{(1-\bar{\alpha} ) a^2+b \left(-\bar{\alpha}  \sqrt{a^2+b^2}+\bar{\alpha}  b+b\right)}{a^2+b^2} & -\frac{a \bar{\alpha}  \left(\sqrt{a^2+b^2}-2 b\right)}{a^2+b^2} \\
 \frac{a \bar{\alpha}  \left(\sqrt{a^2+b^2}-2 b\right)}{a^2+b^2} & \frac{(1-\bar{\alpha} ) a^2+b \left(-\bar{\alpha}  \sqrt{a^2+b^2}+\bar{\alpha}  b+b\right)}{a^2+b^2} \\
\end{array}
\right).
\end{equation}
We know that a linear dynamical system is globally asymptotically stable if and only if the spectral radius of the linear mapping is strictly less than 1. %

Let $\lambda_1,\lambda_2$ be the eigenvalues of $T$.
Then the spectral radius is the largest absolute value of the eigenvalues.
For $T$ this becomes,
\begin{equation}
\max_{i\in\{1,2\}} |\lambda_i| = \sqrt{\frac{(2 (\bar{\alpha} -1) \bar{\alpha} +1) a^2-2 \bar{\alpha}  (\bar{\alpha} +1) b \left(\sqrt{a^2+b^2}-b\right)+b^2}{a^2+b^2}}.
\end{equation}
So we can ask what $c$ in $\rho = -\frac{c}{L}$ needs to be for the sequence $(z^k)_{k\in \mathbb N}$ to converge.
Solving for $c$ in this equality with $\max_i |\lambda_i| < 1$, we obtain,
\begin{equation}
\label{eq:lowerbound-convergence}
c<\frac{1-\bar{\alpha} }{2},
\end{equation}
provided that we pick
\begin{equation}
\label{eq:lowerbound-worstcase}
\frac{\sqrt{1-c^2}}{c}=-\frac{a}{b}.
\end{equation}
Equation \eqref{eq:lowerbound-worstcase} provides a specification for \Cref{ex:lowerbound}.
As long as \eqref{eq:lowerbound-convergence} is satisfied, \eqref{eq:eg+} is guaranteed to converge for $\gamma_k=\nicefrac{1}{L}$.
On the other hand, since \eqref{eq:bilinear-operator} is a linear system, we simultaneously learn that picking $c$ any larger would imply non-convergence through $\max_i |\lambda_i| \geq 1$ (given $z^0 \neq 0$).
We can trivially embed problem \eqref{eq:bilinear-operator} into a higher dimension to generalize the result.
Noting that $c=-\rho L$ completes the proof.
\end{appendixproof}
We provide Mathematica code to verify each step of the above proof.\footnote{\label{footnote:code}The supplementary code can be found at \url{https://github.com/LIONS-EPFL/weak-minty-code/}.}

\section{Toy examples}\label{app:toy-examples}
    \rebuttal{In the following appendix, $L$ denotes the Lipschitz constant of $F$ restricted to the constraint set and $\rho$ is the parameter of the weak MVI (\Cref{ass:Minty:Struct}) when restricted to the constraint set.
This restriction of the definitions is warranted, since $z^k$ remains within the constraint set in all simulations, while $\bar{z}^k$ is guaranteed to stay within by definition of \Cref{state:barz:Struct} in \Cref{alg:WeakMinty:Struct} (and likewise for all other considered method treating problem \eqref{eq:StrucIncl}).}

All computer-assisted calculations can be found in the supplementary code.\textsuperscript{\ref{footnote:code}}

\subsection{Constructing a PolarGame (\texorpdfstring{\Cref{def:polargame}}{\S\ref*{def:polargame}})}\label{app:polargame}

Recall \Cref{def:polargame} which considers a vectorfield $F: \mathbb{R}^n \rightarrow \mathbb{R}^n$ with limit cycles at $r \in \set{c_1,...,c_k}$  where $c_i\neq 0$ for all $i \in [k]$.
Such a vectorfield can be constructed for $n=2$ by departing from the following dynamics in polar coordinates,
\begin{equation}
\label{eq:polar-flow}
\begin{split}
\frac{\partial r}{\partial t} &= -a \cdot r(t) \prod_{i=1}^k (r(t) + c_i) \cdot (r(t) - c_i) \\
\frac{\partial \theta}{\partial t} &= -b \cdot r(t),
\end{split}
\end{equation}
with $a,b \neq 0$.
Transforming this dynamics into cartesian coordinates yields the desired vectorfield, $F$, while subsequently integrating with respect to $x$ and $y$ yields the two potentials associated with the two players.
Note that the roots $\set{-c_i}_{i=1}^k$ for the polynomial defining $\dot{r}$ are not strictly necessary for showing existence of limit cycles, but leads to a simpler form for $Fz$.
We illustrate the construction in \cref{fig:polargame-construction}.

\begin{figure}
\centering
$\vcenterimage{width=0.25\textwidth}{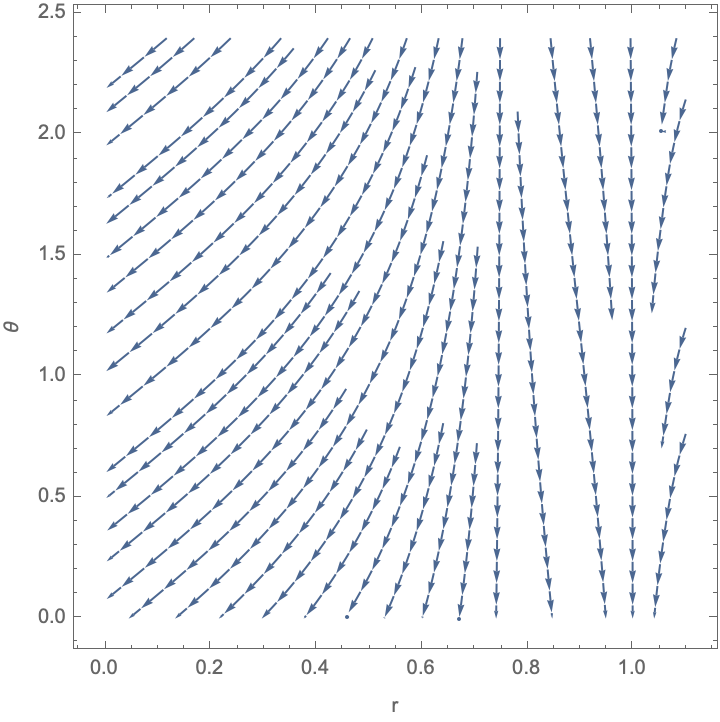}\quad\vcenterarrow\vcenterimage{width=0.25\textwidth}{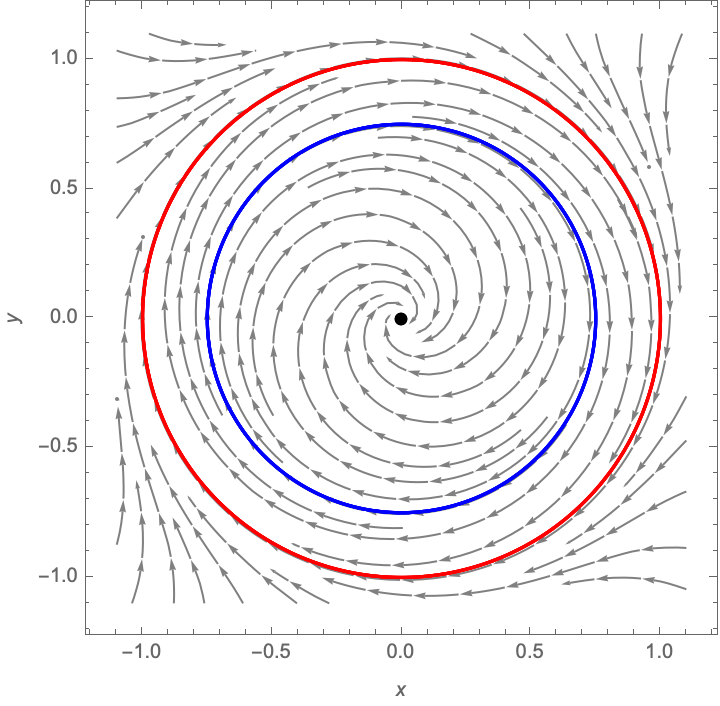}$
\caption{
  We can construct the desired properties in polar coordinates $(r,\theta)$ and subsequently transform it into a vectorfield in cartesian coordinates $(x,y)$.
  This is illustrated by a PolarGame with attracting limit cycles at radius $\|z\|=1$ and repellant limit cycle at $\|z\|=\nicefrac{3}{4}$ for the associated operator $Fz$ as indicated in red and blue respectively. %
}
\label{fig:polargame-construction}
\end{figure}

\begin{proposition}
\label{prop:polargame-limitcycles}
Let $Fz=(\dot{x},\dot{y})$ be the evolution in cartesian coordinates of the associated vectorfield in polar coordinates defined by \eqref{eq:polar-flow}.
Then the only stationary point of $F$ is at the origin $(0,0)$ and there exists a limit cycle at $r=c_i$ for all $i \in [k]$.
\end{proposition}
\begin{proof}
Let $r=\sqrt{x^2+y^2}$.
It is easy to see from \eqref{eq:polar-flow} that the only stationary point is at $r=0$.
By construction, $\dot{r}$ is a polynomial with roots $c_i$ for all $i \in [k]$, so any trajectory starting on the circle defined by $r=c_i$ remains in that set.
However, $\dot{\theta}$ is strictly nonzero.
As a consequence $Fz$ is nonzero, so $r=c_i$ must define a limit cycle, which proofs the claim.
\end{proof}

\begin{figure}[tb]
\centering
\begin{subfigure}[t]{0.31\textwidth}
\caption{$a=1$}
\includegraphics[width=\textwidth]{TeX/Figs/Det_PolarGame_hard.png}
\end{subfigure}%
\begin{subfigure}[t]{0.32\textwidth}
\caption{$a=\nicefrac{3}{4}$}
\includegraphics[width=\textwidth]{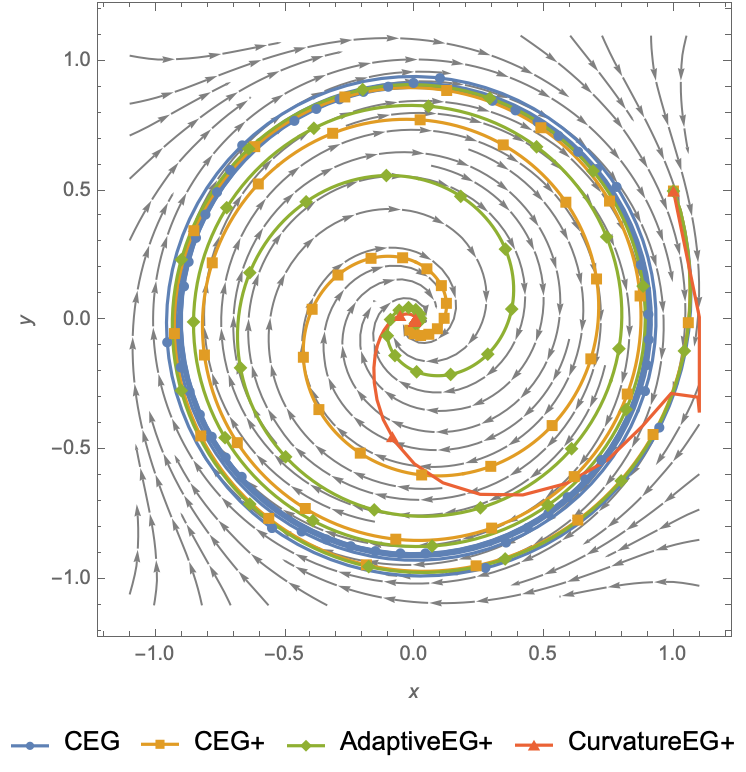}
\end{subfigure}%
\begin{subfigure}[t]{0.31\textwidth}
\caption{$a=\nicefrac{1}{3}$}
\includegraphics[width=\textwidth]{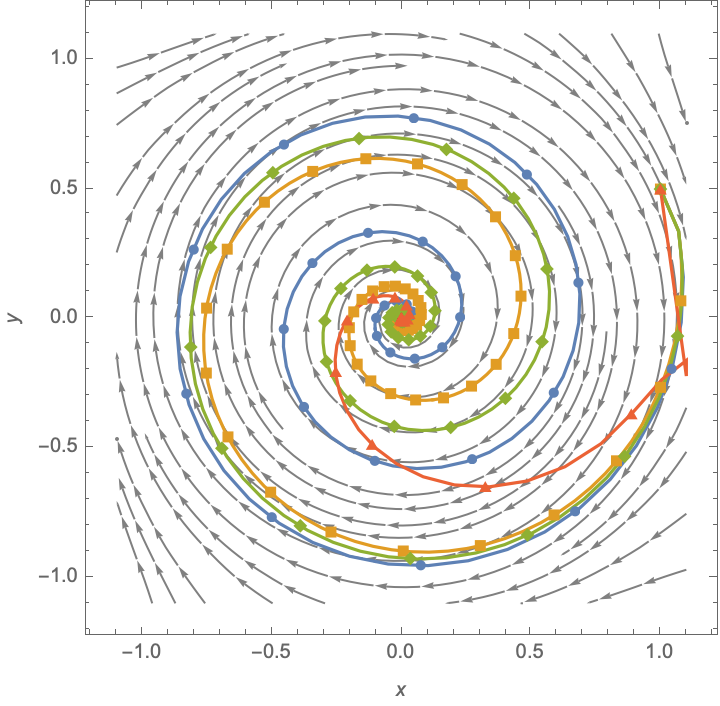}
\end{subfigure}
\caption{
\Cref{ex:polargame} for different values of $a$ (and thereby different values of $\rho$).
Note that even extragradient may escape the limit cycles even though $\rho < 0$. 
This is not in conflict with the negative results of \citet{hsieh2021limits} since the stepsize is not diminishing.
However, in the general case even extragradient with fixed stepsize will not converge as shown by the lower bound in \cref{thm:lowerbound}.
}
\label{fig:polargame}
\end{figure}

\begin{figure}[tb]
\centering
\begin{subfigure}[t]{0.39\textwidth}
\caption{\label{fig:det-extra-a}\Cref{ex:globalforsaken}}
\includegraphics[width=\textwidth]{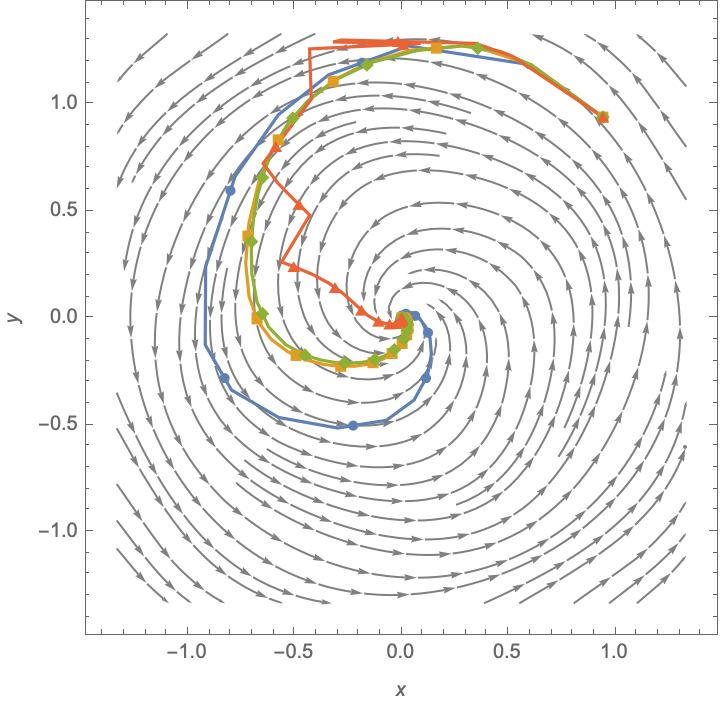}
\end{subfigure}%
\begin{subfigure}[t]{0.38\textwidth}
\caption{\label{fig:det-extra-b}\Cref{ex:lowerbound} ($\rho=\nicefrac{1}{3L}$)}
\includegraphics[width=\textwidth]{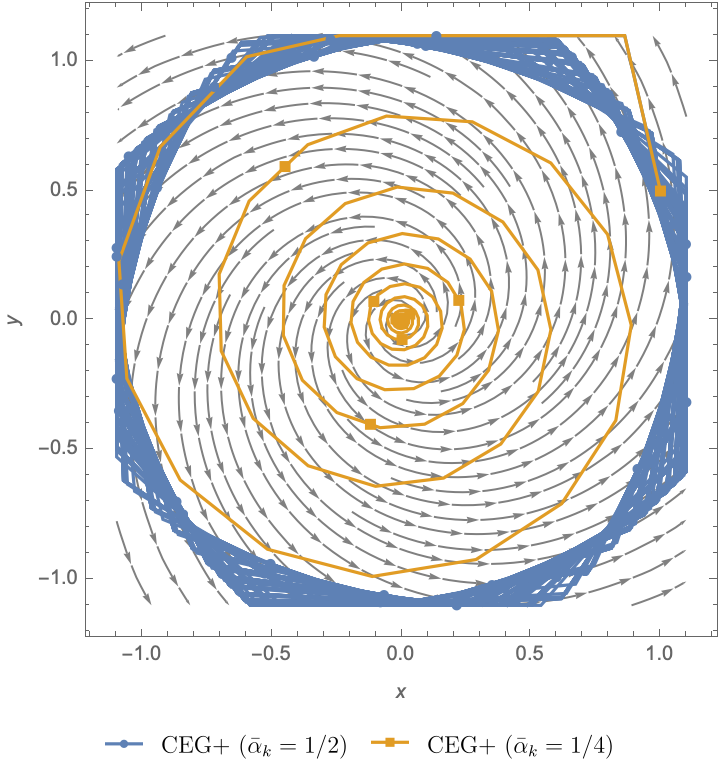}
\end{subfigure}
\caption{In (\subref{fig:det-extra-a}) we observe that all algorithms converge, despite $F$ having an attracting limit cycle in \Cref{ex:globalforsaken}. However, note that in the stochastic setting, where diminishing stepsize is required, SEG does not converge to the critical point (see \cref{fig:stoc-a}). 
In (\subref{fig:det-extra-b}) we demonstrate that when $\rho = -\nicefrac{1}{3L}$, picking $\bar{\alpha}_k < \nicefrac{1}{3}$ for \eqref{eq:iter:constant} is necessary for convergence in general.
See \Cref{sec:experiments} for the experimental setup.
}
\label{fig:globalforsaken-cegplus}
\end{figure}

\subsection{Proof for properties of \texorpdfstring{\Cref{ex:polargame}}{\S\ref*{ex:polargame}}}\label{app:polargame-props}

The operator $F: \mathbb{R}^2 \rightarrow \mathbb{R}^2$ defined in \cref{ex:polargame} is obtained by constructing the associated dynamics in polar coordinates,
\begin{equation}
\begin{split}
\frac{\partial r}{\partial t} &= -a \cdot r(t) \cdot (r(t) + 1) \cdot (r(t) - 1) \cdot (r(t) + \nicefrac{3}{4}) \cdot (r(t) - \nicefrac{3}{4}) \\
\frac{\partial \theta}{\partial t} &= - r(t).
\end{split}
\end{equation}
This can easily be verified by a change of variables.
From \Cref{prop:polargame-limitcycles} it then follows, that there must exist a limit cycle at $\|z\|=1$ and $\|z\|=\nicefrac{3}{4}$.
To verify the conditions on $\rho$ we compute the closed form solution to $\rho$ and $L$ in Mathematica:
\begin{enumerate}
\item 
For $a=1$
we have
$\rho = -\frac{50176}{1050977}$
and
$L=\frac{\sqrt{2538096 \sqrt{704424929}+70246989617}}{20000}$
\item 
For $a=\nicefrac{3}{4}$
we have
$\rho = -\frac{602112}{16798825}$
and
$L=\frac{\sqrt{7614288 \sqrt{6383574361}+635022906553}}{80000}$
\item 
For $a=\nicefrac 13$
we have
$\rho = -\frac{150528}{9439585}$
and
$L=\frac{\sqrt{2538096 \sqrt{754424929}+73446989617}}{60000}$
\end{enumerate}

It can easily be verified that the stated conditions for $\rho$ in \Cref{ex:polargame} are met for the values above.
This completes the proof.

We provide Mathematica code verifying the construction of $F$ and the closed form solutions to $L$ and $\rho$.

\subsection{Proof for properties of \texorpdfstring{\Cref{ex:globalforsaken}}{\S\ref*{ex:globalforsaken}}}
\label{app:GlobalForsaken}

\rebuttal{%
Under the definitions of $\rho$ and $L$ in \Cref{app:toy-examples}, we claim that the origin $(0,0)$ in \eqref{eq:globalforsaken} is a global Nash equilibrium
and satisfies \cref{ass:Minty:Struct} with $\rho > -\nicefrac{1}{2L}$.}

To verify that $(0,0)$ is indeed a global Nash equilibrium we need to check that the solution cannot be unilaterally improved. 
In other words, the solution should coincide with $(x^\star,y^\star)$ where
\begin{equation}
\begin{split}
x^\star &= \argmin_x \phi(x, 0)\\
y^\star &= \argmax_y \phi(0, y).
\end{split}
\end{equation}
We can easily verify this with \texttt{Minimize} in Mathematica, since the functions are polynomial for which a closed form solutions to the global optimization problem will be returned.

To find $\rho$ for $z^\star=(0,0)$ we solve the global minimization problem,
\begin{equation}
\minimize_{z} \frac{\langle Fz,z - z^{\star}\rangle}{\|Fz\|^2},
\end{equation}
for which a closed form solution can be found with Mathematica, which when numerically evaluated is approximately $-0.119732$. 

We need to compute $L$ to ensure $\rho > -\nicefrac{1}{2L}$.
In our case of convex constraints, $\mathcal C$, we have that $L= \sup_{z \in \mathcal C}\|JF(z)\|$ where $\|\cdot\|$ denotes the spectral norm \cite[Thm. 9.2 and 9.7]{rockafellar2009variational}. 
Under our constraint $\|z\|_\infty \leq \nicefrac 43$, this can similarly be computed in closed form, yielding $L = \nicefrac{\sqrt{\frac{1}{2} \left(9409 \sqrt{59721901}+74125591\right)}}{2835}$.
So $-\frac{1}{2L}\approx-0.165432$
which satisfy the condition $\rho > -\frac{1}{2L}$.
This completes the proof.

\begin{proposition}
Let $F$ be the associated operator of $\phi$ in \eqref{eq:globalforsaken} defined as $Fz=(\nabla_x\phi(x,y),-\nabla_y\phi(x,y))$. 
Define the radius as $r=\|z\|$.
Then, $Fz$ has a stable critical point at the origin $(0,0)$, at least one attracting limit cycle in the region defined by $\sqrt{\nicefrac{3}{2
}}<r<2$ and at least one repellant limit cycle within $r \leq \sqrt{\nicefrac{3}{2}}$.
\end{proposition}
\begin{proof}
We follow a similar argument as in \citet[D.2]{hsieh2021limits}.
We can compute the associated operator $F$,
\begin{equation}
\left(
\begin{array}{c}
\dot{x}\\\dot{y}
\end{array}
\right)
=
\left(
\begin{array}{c}
 \frac{4 x^5}{7}-\frac{4 x^3}{3}+\frac{2 x}{3}+y \\
 -x+\frac{4 y^5}{7}-\frac{4 y^3}{3}+\frac{2 y}{3} \\
\end{array}
\right).
\end{equation}
With a change of variables into polar coordinates $(r,\theta)$ we get that $r=\sqrt{x^2+y^2}$ evolves as, %
\begin{equation}
\dot{r} = -\frac{1}{42} r \left(9 r^4 \cos (4 \theta )-14 r^2 \cos (4 \theta )+15 r^4-42
   r^2+28\right).
\end{equation}
When $r = \sqrt{\nicefrac{3}{2}}$ this reduces to $\dot{r}=\frac{3 \cos (4 \theta )+5}{56 \sqrt{6}}$ and we observe that $\dot{r}>0$ for any $\theta$.
Likewise for $r=2$, we have that $\dot{r} = -\frac{4}{21} (22 \cos (4 \theta )+25)$ which implies $\dot{r} < 0$.
Since there is no stationary point in the region $\mathcal S = \set{(r,\theta) : \sqrt{\nicefrac{3}{2}}<r<2}$ it then follows from the Poincar\'e-Bendixson theorem \citep[Thm. 7.16]{teschl2012ordinary} that there must exist at least one attracting limit cycle in $\mathcal S$.
Further, it is easy to see that $(0,0)$ is a critical point and that it is stable by inspection of the Jacobian $JF(z)$.
Since $\mathcal S$ is trapping, it follows from Poincar\'e–Hopf index theorem, that there must exist a repellant limit cycles in the region defined by $r<\sqrt{\nicefrac{3}{2}}$.
This completes the proof.
\end{proof}

\subsection{Proof of properties for \texorpdfstring{\cite[Example 5.2]{hsieh2021limits}}{Forsaken}}
\label{app:Forsaken}

\rebuttal{%
This section considers \cite[Example 5.2]{hsieh2021limits} on the constraint domain $\mathcal D = \{z\in\R^n\mid \|z\|_\infty \leq \nicefrac{3}{2}\}$. 
We show that the unique critical point $z^\star$ does not satisfies the weak MVI for $\rho > -\nicefrac{1}{2L}$ even when restricted to the constraint set $z \in \mathcal D$. 
We restate the example with the additional constraint for convenience.
\begin{example}[{\cite[Example 5.2]{hsieh2021limits}}]
\begin{equation}
\label{eq:forsaken}
\tag{Forsaken}
\minimize_{|x|\leq\nicefrac{3}{2}} \maximize_{|y|\leq\nicefrac{3}{2}} \phi(x,y):=x(y-0.45)+\psi(x)-\psi(y),
\end{equation}
where $\psi(z) = \frac{1}{4} z^{2}-\frac{1}{2} z^{4}+\frac{1}{6} z^{6}$. 
\end{example}

By using Mathematica, we can obtain a closed form solution of the Lipschitz constant $L$ of $F$ restricted to the constraint set, which we find to be $L=\frac{1}{80} \sqrt{\frac{1}{2} \left(1089 \sqrt{801761}+993841\right)}$.
Mathematica can solve approximately for the critical point, yielding $z^\star=(0.0780267, 0.411934)$.
To find $\rho$ we want to globally minimize $\rho(z):=\frac{\left\langle F z, z-z^{\star}\right\rangle}{\|F z\|^2}$ for $z \in \mathcal D$.
Mathematica finds the candidate $z'=(-1.01236, -0.104749)$ for which $\rho(z') = -0.477761$.
So $\rho$ must be at least this small, i.e. $\rho < -0.477761$.
Since $-\nicefrac{1}{2L} \approx -0.04$, this implies that $\rho < -\nicefrac{1}{2L}$.
See \texttt{Forsaken.nb} for Mathematica-assisted computations.

This rules out convergence guarantees for both \eqref{eq:iter:constant} and AdaptiveEG+ (\Cref{alg:WeakMinty:Struct}), which is supported by the simulation in \Cref{fig:forsaken-all}.
However, as observed, \eqref{eq:CurvatureEG} converges in the simulations.

}

\begin{figure}
\centering
\includegraphics[width=0.4\textwidth]{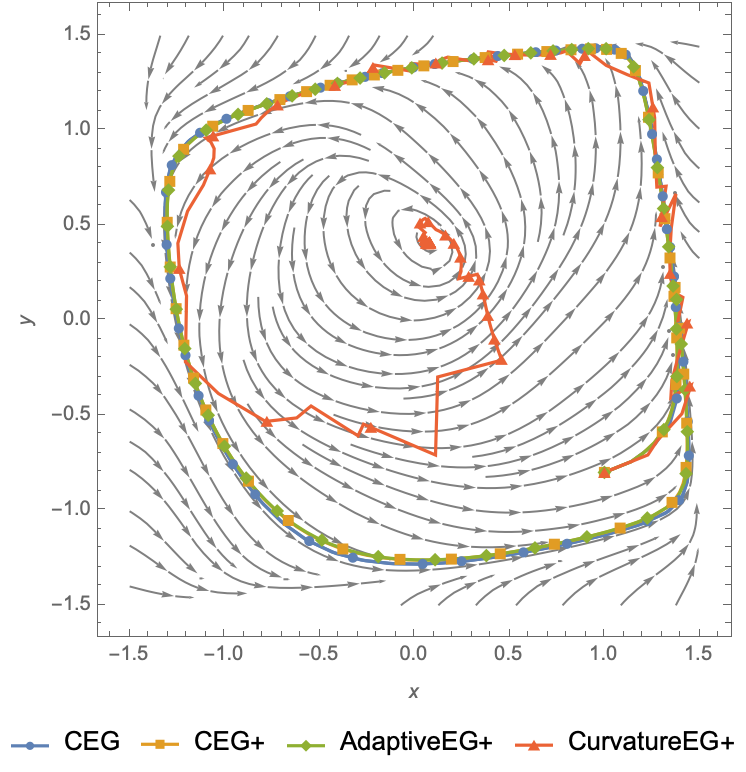}
\caption{Demonstration of algorithms on \cite[Example 5.2]{hsieh2021limits}. Only \eqref{eq:CurvatureEG} converges to the critical point, while
the remaining methods, CEG, \eqref{eq:iter:constant} with $\bar{\alpha}_k = \nicefrac{1}{2}$, and AdaptiveEG+ converges to an attracting limit cycle.
See \Cref{sec:experiments} for further specification of the algorithms.}
\label{fig:forsaken-all}
\end{figure}

\end{document}